\documentclass[12pt]{amsart}
\usepackage{amsthm}
\usepackage{amssymb}
\usepackage{amsmath}
\usepackage{tikz-cd}
\usepackage{graphicx} 

\usepackage[shortlabels]{enumitem}
\usepackage{xypic}
\usepackage{array}
\usepackage{color}

\newtheorem{theorem}{Theorem}[section]
\newtheorem{proposition}[theorem]{Proposition}

\theoremstyle{definition}
\newtheorem{definition}[theorem]{Definition}

\theoremstyle{remark}

\newtheorem{remark}[theorem]{Remark}

\newcommand{\milano}{Dipartimento di Matematica ``F. Enriques"
	\\ Universit\`a degli Studi di Milano \\ Via Saldini 50 \\ 20133
	Milano, Italy}

\newcommand{\Pin}[1]{{\mathbb P}^{#1}}

\newcommand{\oofp}[2]{{\mathcal O}_{\mathbb{ P}^{#1}}({#2})}

\newcommand{\rk}[1]{{\rm rk}\,(#1)}

\newcommand{\bxj}{\{\mathbf{X}_j\}}
\newcommand{\nbXi}{\mathbf{X}_i}
\newcommand{\nbxj}{\mathbf{X}_j}

\begin{document}
	
	\date{\today}
	\title[Problems and results in A.V. and M.G.]{Problems and related results in Algebraic Vision and  Multiview Geometry}
	\author[M.Bertolini]{Marina Bertolini} 
	\email{marina.bertolini@unimi.it} \thanks{The authors are members
		of GNSAGA of INdAM}

	\author[C.Turrini]{Cristina Turrini}
	
	\email{cristina.turrini@unimi.it}
	
	\address{\milano}

 \dedicatory{to our dear friend Enrique}
	
	\begin{abstract}

\noindent The article presents a survey of results in algebraic vision and multiview geometry. The starting points is the study of projective algebraic varieties critical for scene reconstruction. Initially studied for reconstructing static scenes in three-dimensional spaces, these critical loci are later investigated for dynamic and segmented scenes in higher-dimensional projective spaces.
The formal analysis of the ideals of critical loci employs Grassmann tensors, introduced as crucial tools for determining these ideals and aiding the reconstruction process away from critical loci. A long-term goal of the authors with other co-authors involves two main aspects: firstly studying properties of Grassmann tensors, as rank, multi-rank and core, along with the varieties that parameterize these tensors; concurrently conducting an analysis of families of critical loci in various scenarios.
\end{abstract}
	
	
	
	\maketitle
	

\section{Introduction}

This article is a survey of a number of results in algebraic vision and multiview geometry,
two branches of computer vision that are closely related to projective and algebraic geometry.
These results were obtained over a period of several years, in collaboration with different co-authors. 
	
The strand of research on these topics began with the study of special projective algebraic varieties which arise as critical loci for the reconstruction of a static three-dimensional scene, from a set of its two-dimensional images. Images taken by pinhole cameras can be modeled as linear projections from  $ \mathbb{P}^3 $ to $ \mathbb{P}^2$ and original results involved the reconstruction of a set of points, up to projective transformation of the ambient space, starting from two, three, or four such images. Critical loci of scene points are those for which the projective reconstruction fails, in the sense that there exists a non projectively equivalent set of points and cameras giving the same images (\cite{Hart-Zi2},\cite{Hart-Ka}).
	
Later on, the analysis of certain dynamic and segmented scenes led computer vision experts to consider higher dimensional projective spaces, given that  a dynamic or segmented scene in $ \mathbb{P}^3 $ can be modeled as a static and more manageable scene in a higher dimensional ambient space  $ \mathbb{P}^k , k\ge 3$ (see, e.g. \cite{WS}). Thus, the classical problem of projective reconstruction hes been generalized in the setting of high dimensional projective spaces and the notion of criticality has been naturally extended as well to projections from $ \mathbb{P}^k $ to image spaces of higher dimension, $ \mathbb{P}^h, h\ge 2.$ The resulting generalized critical loci turn out to be very classical and well known projective algebraic varieties, arising as determinantal varieties (\cite{tubbLAIA},\cite{tubbCHAPTER}, \cite{beMagri}, \cite{bnt2}, \cite{bnt1}, \cite{bbnt}).
	
	The study of the ideal of critical loci has been formalized using the so-called Grassmann tensors (or multiview tensors) introduced in \cite{Hart-Schaf}.
	These tensors are defined for a set of projections from $ \mathbb{P}^k $ to $ \mathbb{P}^{h_j}, j=1,\dots,n$. They encode multilinear constraints among view spaces $\mathbb{P}^{h_j},$ obtained considering various relationships among corresponding linear subspaces, i.e. subspaces that contain images of the same scene point in the different target spaces.
	
	These tensors are a key tool not only to determine the ideal of critical loci, but also to carry out the reconstruction process of a scene whose points are far enough from the critical loci. 
	In this framework, one of  the authors' and their coauthors' long term goals is to investigate fundamental properties of Grassmann tensors including various notions of rank, decomposition, degenerations, and identifiability  in higher dimensions, and, when feasible, to study the varieties parametrizing such tensors. Results contained in \cite{BBBT1}, \cite{bebitu}, and \cite{tubbAMPA} are to be framed in this context.

Concerning some bibliographic references, a foundational book on multiview geometry is \cite{Hart-Zi2}. Multiview varieties are not addressed in this paper, but are strictly connected with these arguments (\cite{ah-st-th}, \cite{THP}). For tensors and tensor decomposition we refer to \cite{LA}, and for some applications to \cite{be-car-cat-gi-on}, while for a wide study of the multifocal tensor varieties we refer to \cite{oe1}, \cite{oe2} and \cite{hey1}. A recent survey with updated bibliography on Algebraic Vision is given in \cite{Snapshot}.

\
	
The paper is structured as follows. Assuming no previous knowledge of computer vision, in Section \ref{CV} we recall the setting and standard definitions of multiview geometry and we give some examples to show how a dynamic scene in the three dimensional space can give rise to a static one in a higher dimensional space. In Section \ref{Grass-tens} we introduce Grassmann tensors, following \cite{Hart-Schaf}, and we show their role in the reconstruction process. In Section \ref{rank-core} the focus is on the case of two or three pojections and we study the properties of the corresponding Grassmann tensors, classically known as the bifocal and trifocal tensor, respectively. We determine the rank of the bifocal tensor (which is indeed a matrix), the rank, multilinear rank and the core of the trifocal tensor, under a generality assumption on the centers of projections. This natural assumption allows us to obtain a canonical form for the tensor that greatly simplifies all related computations.
In Section \ref{var-grass} we introduce the variety of Grassmann tensors and we give some results on its birational geometric properties in the bifocal case.

Section \ref{GTCL} opens the second part of the paper, dedicated to the study of critical loci. In particular, in Section \ref{GTCL} we show how the Grassmann tensors can be used to properly define critical loci and to compute their ideals. 
In Sections from \ref{CALIB} to \ref{DEGENERATE} we study critical varieties in different contexts. More precisely, in Section \ref{CALIB} we consider critical loci for the calibration of a camera and in Section \ref{MORE2} critical varieties in the case of at least two projections, when they turn out to be smooth. In Section \ref{BORDIGA} we study the particular case of three projections from $ \mathbb{P}^4 $ to $ \mathbb{P}^2$ which produces a Bordiga surface, connected with a line congruence in $G(1,3)$. In Section \ref{DEGENERATE} we briefly recall what happens in the degenerate cases. In Section \ref{UNIFIED} we introduce and give some results on the unified critical loci and finally in Section \ref{INSTABILITY} we recall the main steps of the practical experiments, which prove how the reconstruction near a critical locus can be unstable.
	
Throughout the paper, proofs are omitted or only briefly sketched, with two exceptions: the explicit construction of the canonical form for the bifocal and trifocal tensors and the definition and properties of critical loci via Grassmann tensors. This is due to the central role that these constructions play in the rest of the paper.

\section{Reconstruction problems}
\label{CV}

One of the most classical, still challenging tasks in Computer Vision, specifically in Multiview Geometry, is the so-called {\it{Structure from Motion}} or {\it{Reconstruction}} problem. In this context, when presented with a set of images (photos or videos) capturing the same scene, whether static or dynamic, the objective is both to reconstruct
\begin{itemize}
    \item  the position and all the data of the cameras that captured the images 
\end{itemize}
and 
\begin{itemize}
    \item  the spatial coordinates of the points in the scene and all additional data such as movement-related information, including velocities and more.
\end{itemize}

The first task is called {\it{calibration}} of the cameras, the second is the actual reconstruction of the scene.

Obviously this problem can be addressed at many different levels, depending on the type of reconstruction you want to obtain: up to homography, up to affinity, or up to similarity. Here we will only deal with the most generic reconstruction, i.e. the projective one, hence we will not take into account all the metric aspects of the set up. 

\subsection{Classical set up: a static scene and, at least, two cameras}

We start recalling the standard set up in multiview geometry.
A {\it{scene}} is a set of N points $X \in \mathbb{P}^3$. A "pinhole" camera is represented as a central projection $P$ of points in $3$-space, from a point $C$ (the {\it{center of projection}}) onto a suitable target plane $\mathbf{P}^2$, which is called {\it{view}}.

After choosing appropriate bases in the vector spaces involved, we can identify
the projection map $P$ with its representative matrix, for which we use the
same symbol $P$. Accordingly, if $X$ is a point in $\mathbb{P}^3$, we denote its image in the projection equivalently as $P(X)$ or $PX$. The line through $C$ and $X$ is a {\it{ray.}} 

\

Of course, since all points of the same ray have the same image, the task of reconstructing the scene is impossible, starting from only one view. The only possible reconstruction is calibration. Moreover, since we are dealing with reconstructions up to projective transformations, the center is the only property of the camera that can be reconstructed.

\ 

Hence, for the reconstruction of the scene, at least two views are needed. 
Points on the views that are images of the same point of the scene are said to be {\it{corresponding points}}.
		
A well known result shows that, if enough pairs of corresponding points, in general position, are known, the reconstruction process is successful. 
The key tool in reconstruction process is a matrix (the {\it{essential matrix}} introduced in 1981 by Christopher Longuet-Higgins, equivalently the {\it{fundamental matrix}}) which encodes the possible constraints that exist between two different images of the same scene.

\subsection{Dynamic scenes}

Experts in computer vision have shown interest in linear projections, not only $P: \mathbb{P}^3 \dashrightarrow \mathbb{P}^2,$ but also between higher dimensional projective spaces. The reason for this lies in the possibility for these projections to model shots of dynamic scenes, such as videos or film sequences.
The idea (due to L. Wolf and A. Shashua \cite{WS}) is that of replacing a dynamic scene that takes place in the $3$-dimensional space with a static scene in a $k$-dimensional space where coordinates represent the initial position of points and the data of the movement (speed, acceleration, ...).

Of course in this context the problem of reconstruction becomes that of recovering all the data of the dynamic scene: the coordinates of the initial points, their speeds, accelerations, and so on.

\

Here is a couple of examples of this approach.

\

\noindent {\bf{$3$-dimensional dynamic scene with constant velocities and parallel trajectories}}

\noindent Assume you have a set of points in the 3D space where each point  $X_j$ moves independently along a straight-line path with constant velocity $\lambda_j$ and such that all the trajectories are parallel to each other. Let $P(t)$ denote the projection matrix from $\mathbb{P}^3$ to $\mathbb{P}^2$ at time $t$ and let $P_j(t)$ be the $j-$th column of $P(t)$. Assume that the common direction of the trajectories is given by the unit vector $(dx_0: dy_0: dz_0)$. If $X_j$ starts at $(x,y,z)^T$ at time $t=0$, then the position of $X_j$ at time $t$ is given by $X_j = (x+t\lambda_jdx_0, y+t\lambda_jdy_0,z+t\lambda_jdz_0,1)$. One can embed
this scene in $\mathbb{P}^4$, with coordinates $(X_1:X_2:X_3:X_4:X_5)^T$ by setting
$X_1 = x; X_2 = y; X_3 = z; X_4 = 1; X_5 = \lambda_j)$.
Denote by $P$ the following $3 \times 5$ matrix
$$P = [P_1(t) | P_2(t) | P_3(t) | P_4(t) | tdx_0 P_1(t) + tdy_0 P_2(t) + t dz_0 P_3(t)].$$
It is immediated to see that one has 
$$P (x,y,z,1,\lambda_j) = P(t) X_j,$$
hence the dynamic scene in $\mathbb{P}^3$ has been trasnformed into a static one in $\mathbb{P}^4.$

   \

\noindent {\bf{$3$-dimensional dynamic scene with constant velocities and no restriction on the trajectories}}

\noindent Assume you have a set of points in the 3D space where each point moves independently along some straight-line path with no restriction on the mutual positions of the trajectories. Let $P(t)$ denote the projection matrix from $\mathbb{P}^3$ to $\mathbb{P}^2$ at time $t$ and let $P_j(t)$ be the $j-$th column of $P(t)$. If point $X_i$ starts at $(x: y: z)$ at time $t = 0$; and its velocity vector is $(dx: dy: dz)$; then the homogeneous coordinates of
point $X_i$ at time $t$ are $X_i = (x+tdx: y +tdy: z +tdz: 1)^T$ . One can embed
this scene in $\mathbb{P}^6$, with coordinates $(X_1:X_2:X_3:X_4:X_5:X_6:X_7)^T$ by setting
$X_1 = x; X_2 = y; X_3 = z; X_4 = 1; X_5 = dx; X_6 = dy; X_7 = dz)$.
The projection matrix from $\mathbb{P}^6$ to $\mathbb{P}^2$ is then the $3 \times 7$ matrix
$$[P_1(t) | P_2(t) | P_3(t) | P_4(t) | tP_1(t) | tP_2(t) | tP_3(t)].$$

\section{Grassmann tensors}
\label{Grass-tens}

So, from now on we will deal with linear projections  $P: \mathbb{P}^k \dashrightarrow \mathbb{P}^h$ defined by $(h+1) \times (k+1)$ full rank matrices. 

\

To proceed further, it is essential to introduce specific notation.

Let $\mathcal{V}$ be a $k+1$-dimensional vector space. We will denote by $\mathbb{P}^k = \mathbb{P}(\mathcal{V})$ the projective space of one-dimensional subspaces of $\mathcal{V}.$
Fix a proper subspace $\mathcal{C} \subset \mathcal{V}$ of dimension $k-h$ ($h < k$), and consider the quotient map $p: \mathcal{V} \rightarrow \mathcal{V}/\mathcal{C}$. The map $p$ induces a rational map $P:\mathbb{P}(\mathcal{V}) \dashrightarrow \mathbb{P}(\mathcal{W})$ where $\mathcal{W} = \mathcal{V}/\mathcal{C}$.

We will call $P$ ${\it{camera}}$ and $C=\mathbb{P}(\mathcal{C})$ {\it{center of projection.}} The target space $\mathbb{P}(\mathcal{W})$ is called {\it{space of rays}}: a point of $\mathbb{P}(\mathcal{W})$ can be identified with a projective $k-h$ linear space containing the center of projection and will be called {\it{ray}}.

As usual $\mathbb{P}({\mathcal{V}}^\star)$ will be identified with the linear space of hyperplanes of $\mathbb{P}(\mathcal{V})$ and $\mathbb{P}({\mathcal{W}}^\star)$ with the subspace od hyperplanes containing $C$. We will call $\mathbb{P}({\mathcal{W}}^\star)$ the {\it{view space}}.

Once bases in $\mathcal{V}$ and in $\mathcal{W}$ are fixed, one obtains a representative matrix for $P$ which we will still denote by $P$. Columns of $P$ generate $\mathcal{W}$, whilst rows of  $P$  generate ${\mathcal{W}}^\star$.

\

We are now ready to introduce Grassmann tensors.
		
Given a set of $n$ linear projections $P_j: \mathbb{P}^k \dashrightarrow \mathbb{P}^{h_j}$, we choose a {\it{profile}} i.e. a $n-$tuple of positive integers $(\alpha_1, \dots \alpha_n)$ such that $\alpha_1 + \dots + \alpha_n = k+1$.

One can choose a linear subspace $L_j \subset \mathbb{P}^{h_j}$  of codimension $\alpha_j$ in the $j-$th view, $j = 1, \dots n$. The subspaces $(L_1, \dots L_n)$ are said to be {\it{corresponding}} if there exist at least one point  $X \in \mathbb{P}^k$ such that $P_j(X) \in L_j, \ \forall j = 1, \dots n$.
						
Let us consider a matrix $S_j$ whose columns represent a basis of $L_j$.
			
The condition $P_j( X) \in L_j$ translates into $P_j X = S_j \mathbf{v_j}$, for some $\mathbf{v_j} \in \mathbb{C}^{h_j-\alpha_j+1}$.
			
In other words, if $(L_1, \dots L_n)$ are corresponding subspaces, the following square homogeneous system has non trivial solution
			
\begin{equation}
\label{syst}
\left[
\begin{array}{ccccc}
P_1 & S_1 & \mathbf{0} & \dots & \mathbf{0} \\
P_2  & \mathbf{0}  & S_2 & \dots & \mathbf{0} \\
\dots & \dots & \dots & \dots & \dots \\
\dots & \dots & \dots & \dots & \dots \\
P_n  & \mathbf{0}  & \mathbf{0} & \dots & S_n \\
\end{array} \right] \cdot \left[
\begin{array}{c}
X \\
-\mathbf{v}_1 \\
-\mathbf{v}_2 \\
\vdots \\
-\mathbf{v}_n\\
\end{array}
\right] = \left[
\begin{array}{c}
\mathbf{0} \\
\mathbf{0} \\
\mathbf{0} \\
\vdots \\
\mathbf{0}\\
\end{array}
\right], 
\end{equation} 
so that $det(T^{P_1, \dots,
P_n}_{S_1, \dots, S_n}) = 0,$			
\quad where $T^{P_1, \dots,
P_n}_{S_1, \dots, S_n} = \left[
\begin{array}{ccccc}
P_1 & S_1 & \mathbf{0} & \dots & \mathbf{0} \\
P_2  & \mathbf{0}  & S_2 & \dots & \mathbf{0} \\
\dots & \dots & \dots & \dots & \dots \\
\dots & \dots & \dots & \dots & \dots \\
P_n  & \mathbf{0}  & \mathbf{0} & \dots & S_n \\
\end{array} \right]  $
			
\

Notice that $T^{P_1, \dots,
P_n}_{S_1, \dots, S_n}$ is an $n-$linear form in the Pl\"ucker coordinates of $L_j'$s, hence it defines a tensor  $\mathcal{T}$ which is called {\it{Grassmann tensor}} defined by the $P_j'$s and the profile $(\alpha_1, \dots, \alpha_n)$ and which, in the sequel, we will also denote by $\mathcal{T}^{(P_1, \dots P_n)}$, or $\mathcal{T}^{(P_1, \dots P_n)}_{(\alpha_1, \dots \alpha_n)}$, when necessary. 

\noindent The entries $\tau_{j_1,\dots,j_n}$  of $\mathcal{T}$ are $(k+1) \times (k+1)-$ minors of $ \left[
				\begin{array}{c}
					P_1  \\
					\vdots  \\
					P_n  \\
				\end{array} \right]$ obtained by choosing $\alpha_j$ rows of $P_j$. Hence, in the decomposition
				
				$$\bigwedge^{k+1}({{\mathcal{W}}_1^\star} \oplus \dots \oplus {{\mathcal{W}}_n^\star}) = \bigoplus_{\sum r_j= k+1} \bigwedge^{r_1}{{\mathcal{W}}_1^\star} \otimes \dots \otimes  \bigwedge^{r_n}{{\mathcal{W}}_n^\star}$$
				
\noindent they correspond to  $\bigwedge^{\alpha_1}{{\mathcal{W}}_1^\star} \otimes \dots \otimes  \bigwedge^{\alpha_n}{{\mathcal{W}}_n^\star}$; in other words $\mathcal{T} \in V_1 \otimes  V_2 \otimes \dots \otimes V_n,$ where $V_j = \bigwedge^{\alpha_j}{{\mathcal{W}}_j^\star}.$

\

These tensors have been introduced by Hartley and Schaffalitsky in (\cite{Hart-Schaf}) and play a key role in Structure-from-motion problems due to the following fundamental result. 

\begin{theorem} \label{fund_result} (see \cite{Hart-Schaf})
If at least one $j =1, \dots,n$ is such that $h_j > 1$, the Grassmann tensor $\mathcal{T}$ determines the projection matrices $P_1, \dots, P_n$, up to constants and up to a projective transformation in $\mathbb{P}^k$.
\end{theorem}

\noindent It is useful to notice that if all the target spaces are $\mathbb{P}^1$ the reconstraction is always ambiguous, see for example \cite{Hart-Schaf} and \cite{IMU}.

\
			
\subsection{Classical multiview tensors}
\label{class_ex}

Grassmann tensors generalize to linear projections between spaces of higher dimension some tools which where well-known in the case of projections from $\mathbb{P}^3$ to $\mathbb{P}^2$:

\begin{itemize}
\item the {\it{fundamental matrix}}, in the case of two projections from $\mathbb{P}^3$ to $\mathbb{P}^2$, with profile $ (\alpha_1,\alpha_2) = (2,2)$; 
\item the {\it{trifocal tensor}}, in the case of three projections from $\mathbb{P}^3$ to $\mathbb{P}^2$, with profile $(\alpha_1,\alpha_2,\alpha_3) = (1,1,2)$ ;
\item the {\it{quadrifocal tensor}}, in the case of four projections from $\mathbb{P}^3$ to $\mathbb{P}^2$, with profile $(\alpha_1,\dots,\alpha_4) = (1,1,1,1)$.
\end{itemize}

\

Here we limit ourselves to mentioning the main properties of the fundamental matrix and of that particular geometry (called {\it{epipolar geometry}}) studying correspondences between two views. For further information on this topic and for an in-depth analysis of the case of the trifocal and quadrifocal tensor we refer to \cite{Hart-Zi2}.

\

Consider two linear projections $P_1, P_2: \mathbb{P}^3 \dashrightarrow \mathbb{P}^2,$ and denote by $C_1, C_2 \in \mathbb{P}^3$ the centers of $P_1$ and $P_2$ (resp.). The images $E_{i,j} = P_i(C_J)$ ($i,j = 1,2,  i \neq j$) of the center $C_j$ in the $i-$th view, via the $i-$th projection, are called {\it{epipoles}}.

For a given $\overline{\mathbf{x}}$ in the first view, the set of points in the second view that are corresponding to $\overline{\mathbf{x}}$ is a line 
${\lambda}_{\overline{\mathbf{x}}}$ through the epipole $E_{2,1}$, and similarly the set of points in the first view which are corresponding to a fixed $\overline{\mathbf{x}'}$ is a line ${\lambda}_{\overline{\mathbf{x}'}}$ passing through $E_{1,2}$. Lines ${\lambda}_{\overline{\mathbf{x}}}$ and ${\lambda}_{\overline{\mathbf{x}'}}$ are called {\it{epipolar lines}}.

As we have seen above, the fundamental matrix $F$ is defined by $$ det(M) = det \left[
			\begin{array}{ccc}
				P_1 & \mathbf{x} & \mathbf{0} \\
				P_2  & \mathbf{0}  & \mathbf{x}' \\
			\end{array} \right] = (\mathbf{x}')^t F \mathbf{x}$$

\noindent and  if $\mathbf{x} = P_1(\mathbf{X})$ and $\mathbf{x}' = P_2(\mathbf{X})$ is a pair of corresponding points, then $(\mathbf{x}')^t F \mathbf{x} = 0$. 

In the sequel (see Section \ref{rank-core}) we will deal with the ranks of multiview tensors. In the case of the fundamental matrix it is very easy to see that it is a $3 \times 3$ matrix of rank $2$.
Indeed the epipolar lines ${\lambda}_{\overline{\mathbf{x}}}$ and ${\lambda}_{\overline{\mathbf{x}'}}$ have resp. equations $(\mathbf{x}')^t F \overline{\mathbf{x}} = 0$ and
${\overline{\mathbf{x}'}} F \mathbf{x} = 0$, and the epipoles are the right and left annihilators of $F$. This implies that $\rk{F}=2.$

\

\subsection{Steps for the reconstruction}
\label{reconstr}

We mentioned in the introduction that Grassmann tensors were introduced to perform the reconstruction process in the general context of any number of projections between spaces of any dimension. In this section, we will show the step-by-step solution to the structure-from-motion problem.

We confine our discussion to the theoretical aspects of the question. In the real world, all data is affected by noise, so it is necessary to treat the problem with Numerical Analysis tools, but we will not address this aspect of the question. 

\

Assume we have a set of $n \geq 2$ linear projections $P_j: \mathbb{P}^k \dashrightarrow \mathbb{P}^{h_j}$, $j=1, \dots n$ which are unknown and which project unknown scene points $\mathbf{X} \in \mathbb{P}^k$. Assume also that we can identify some corresponding points in the images. 

We choose a profile $(\alpha_1, \dots \alpha_n)$, with $\alpha_1 + \dots + \alpha_n = k+1$ and the associated Grassmann tensor

$$ det(T^{P_1, \dots,
P_n}_{S_1, \dots, S_n}) = det \left[
				\begin{array}{ccccc}
					P_1 & S_1 & \mathbf{0} & \dots & \mathbf{0} \\
					P_2  & \mathbf{0}  & S_2 & \dots & \mathbf{0} \\
					\dots & \dots & \dots & \dots & \dots \\
					\dots & \dots & \dots & \dots & \dots \\
					P_n  & \mathbf{0}  & \mathbf{0} & \dots & S_n \\
				\end{array} \right].$$

\begin{itemize}
\item [1]  {\bf{reconstruction of the tensor}}:

\noindent For any set $(L_1, \dots L_n)$ of corresponding spaces, one has to compute their Pl\"ucker coordinates and to substitute them in the expression of $det(T^{P_1, \dots,
P_n}_{S_1, \dots, S_n})$. In this way, $det(T^{P_1, \dots,
P_n}_{S_1, \dots, S_n}) = 0$ turns out to be a linear equation in the unknowns entries $\tau_{j_1,\dots,j_n}$ of $\mathcal{T}$. 

If sufficiently many $n-$tuples of corresponding subspaces in the views (and in "enough general" position) are recognized, one gets a linear system from which the entries of the tensor can be obtained. Of course, in the real word case, one has to deal a with overdetermined linear system and to solve it with numerical methods.

\end{itemize}
						
\begin{itemize}
\item [2] {\bf{ reconstruction of the cameras or calibration}}: 
				
\noindent Once the Grassmann tensor $\mathcal{T}$ has been reconstructed, Theorem \ref{fund_result} is used to determine projection matrices, up a projective trasnsformation in $\mathbb{P}^k$ and up to constants.

\item [3] { \bf{reconstruction of the scene}}:

\noindent Finally, once the projection matrices are reconstructed, for each $n-$tuple of corresponding subspaces $(L_1, \dots, L_n)$, one can recover the scene point $\mathbf{X}$ such that $P_j(\mathbf{X}) \in L_j, j=1, \dots, n$, by solving the linear system (\ref{syst}). 
							
\end{itemize}

\section{Ranks and Core of Grassmann tensors}
\label{rank-core}

In this section we recall the well-known notions of rank and multilinear rank of a tensor. For a text that deals extensively with this type of topic, see, e.g. \cite{LA}. We also mention the notion of core of a tensor (see \cite{tu}). Finally we study the ranks and the core in the case of the Grassmann tensor.

\subsection{Ranks of tensors}

Let $Z_j$ be a vector space of dimension $dim Z_j = n_j, \ j=1, \dots h$ and consider a tensor $\mathcal{K} \in Z_1 \otimes  Z_2 \otimes \dots \otimes Z_h.$

One can consider different notions of "rank" for $\mathcal{K}$. The most standard is the following.
			
The {\it{rank}} of $\mathcal{K}$, denoted by rk$({\mathcal{K}}) $ is the minimum number of decomposable tensors $ \mathbf{z}_1 \otimes  \mathbf{z}_2 \dots \otimes \mathbf{z}_h $ ($\mathbf{z}_j \in Z_j$) needed to write $\mathcal{K}$ as a sum.
			
\
			
Another important notion of rank for a tensor, is the so-called multilinear rank. To introduce this notion we have to recall the definition of flattening matrices. We will explain it explicitly in the case of interest to us, i.e. in the case of a three dimensional tensor.

Let $Z_1,Z_2,Z_3$ be vector spaces of dimension $n_1,n_2,n_3,$ with chosen bases $\{\alpha_i\}, \{\beta_j\}, \{\gamma_k\},$ respectively.

Let $\mathcal{K}=[K_{i,j,k}]\in Z_1 \otimes Z_2
\otimes Z_3$. We can interpret $Z_1 \otimes Z_2 \otimes Z_3$ as $Z_1
\otimes (Z_2 \otimes Z_3)$ to get
\begin{equation}
	\label{TA} \mathcal{K}=\sum_{i}\alpha_i \otimes
	(\sum_{j,k}K_{i,j,k}(\beta_j \otimes \gamma_k)).
\end{equation}
The corresponding matrix, of size $n_1 \times (n_2 n_3)$, which
is the flattening $\mathcal{K}_1$, and has the following block
structure:

\begin{equation*}
	{\tiny
		\begin{bmatrix}
			K_{1,1,1}& \dots & K_{1,n_2,1}&K_{1,1,2} & \dots & K_{1,n_2,2}& \dots &K_{1,1,n_3}& \dots & K_{1,n_2,n_3} \\
			K_{2,1,1}  & \dots & K_{2,n_2,1}&K_{2,1,2}& \dots & K_{2,n_2,2}& \dots &T_{2,1,n_3}& \dots & K_{2,n_2,n_3}\\
			\  & \vdots & \ & \   & \vdots &  \ & \dots & \  & \vdots & \\\
			K_{n_1,1,1}  & \dots & K_{n_1,n_2,1}&K_{n_1,1,2}& \dots & K_{n_,n_2,2}& \dots &K_{n_1,1,n_3}& \dots & K_{n_1,n_2,n_3}\\
	\end{bmatrix}}
\end{equation*}
In the same way, paying attention to the cyclic nature of indices $i,j,k,$ one can define flattenings $\mathcal{K}_2$ and
$\mathcal{K}_3.$ One then defines the {\it mutilinear} rank (or  F-rank) of the tensor $\mathcal{K}$ as F-$\rk{\mathcal{K}} =(\rk{\mathcal{K}_1},\rk{\mathcal{K}_2},\rk{\mathcal{K}_3}).$

In the case of an $h$-dimensional tensor $\mathcal{K} \in Z_1 \otimes  Z_2 \otimes \dots \otimes Z_h,$ we proceed in a similar way to construct the $a-$th flattening $\mathcal{K}_a$ ($a=1, \dots n$) which turns out to be a $n_a \times (n_1  \dots \check{n_a} \cdots n_h)$ matrix (where $\check{(-)}$ denotes deleting). One defines the {\it{multilinear rank}} (or {\it{flattening rank}} ) of $\mathcal{K}$ to be F-rk$({\mathcal{K}}) = (r_1,r_2, \dots r_h),$ where $r_a =$rk$({\mathcal{K}}_a)$,  is the rank of the $a-$th flattening matrix ${\mathcal{K}}_a.$  

\

Of course, if $n=2,$ the tensor $\mathcal{K}$ turns out to be a matrix $M$, its flattening are $\mathcal{K}_1 = M$ and $\mathcal{K}_2 = M^T$, so that one has F-rk$({\mathcal{K}}) = (r,r),$ where $r = \rk{{\mathcal{K}}} = \rk{M}.$
\

\begin{remark} 
	\label{invariance} 
For what follows it is worth observing that
	\begin{itemize}
		\item the rank of a tensor $\mathcal{K} \in Z_1 \otimes  Z_2 \otimes \dots \otimes Z_h,$ is invariant under multilinear multiplication by $(M_1,\dots, M_h)$ where $M_a \in GL(n_a), a = 1, \dots, h;$
		\item the rank of ${\mathcal{K}}_a$, $a=1, \dots h,$ is invariant under left action of $GL(n_a)$ and the right action of  $GL(n_{a+1}) \otimes \dots \otimes GL(n_h) \otimes GL(n_1) \otimes \dots \otimes GL(n_{a-1}) ).$
	\end{itemize}
	
\end{remark}
	
\ 
			
Finally we recall the notion of core of a tensor. The core of a tensor is crucial for applications as it represents a smaller tensor that retains all the information of the original tensor.

Let $\mathcal{K} \in Z_1 \otimes  Z_2 \otimes \dots \otimes Z_h$, with $dim Z_j = n_j, \ j=1, \dots h$ be a tensor with multilinear rank F-rk$({\mathcal{K}}) = (r_1,r_2, \dots r_h).$
			
By a {\it{core tensor}} of ${\mathcal K}$  we mean a tensor ${\mathcal C}$ satisfying:
\begin{itemize}
\item[1.] ${\mathcal C} \in W_1 \otimes \dots \otimes W_h$, where $dim W_j = r_j,  \ j=1,\dots h $ ;
\item[2.] there exist $n_j \times r_j$ semi-orthogonal matrices $U_j$, (i.e. $U_j^*U_j = I_{r_j}$, where $(-)^*$ denotes the adiont matrix), such that:  
$$(U^*_1,\dots ,U^*_h) \cdot \mathcal{K} = \mathcal{C};$$
$$(U_1, \dots ,U_h) \cdot {\mathcal C} = \mathcal{K}.$$ 
\end{itemize}

\subsection{Ranks of Grassmann tensors: classical cases}

In this section we study ranks and multilinear ranks of the Grassmann tensors.

First of all, it it is worth noting that the study of ranks of Grassmann tensors, is interesting not only from a teorical point of view, but also for applications. To explain it let's consider the simplest scenario: that of $2$ projections from $\mathbb{P}^3$ to $\mathbb{P}^2$. In this case, the tensor corresponds to the $3 \times 3$ fundamental matrix. The knowledge of  this matrix is a crucial step in reconstructing the scene. However, in practical applications, when the fundamental matrix is reconstructed from a set of pairs of corresponding points, the resulting matrix $\tilde{F}$ is an approximate version of the "true" fundamental matrix $F$ and typically has a rank of $3$
As highlighted in section \ref{class_ex}, $F$ is of rank $2$, so that a refinement strategy involves projecting the reconstructed matrix $\tilde{F}$ onto the variety of matrices of rank $2$.

\

In the classical cases of section \ref{class_ex}, (i.e. for n=2, 3, or 4 projections $ \mathbb{P}^3 \dashrightarrow \mathbb{P}^2$) ranks of Grassmann tensors	are well known, see for example, \cite{Hart-Zi2}, \cite{hey1}:
\begin{itemize}
\item for $n=2$ ({\it{fundamental matrix}}), rk$(\mathcal{T}) =2$;
\item for $n=3$ ({\it{trifocal tensor}}),  rk$(\mathcal{T}) = 4$;
\item for $n=4$ ({\it{quadrifocal tensor}}),  rk$(\mathcal{T}) = 9$.
\end{itemize}
			
\subsection{Bifocal and trifocal tensors: canonical form}

The authors, together with some co-authors, have started to study the ranks of Grassmann tensors for the higher dimensional cases, i.e. for projections $\mathbb{P}^k \dashrightarrow \mathbb{P}^{h_j}, j=1, \dots,n$, in case $n=2$ ({\it{bifocal tensor}}) (\cite{tubbAMPA}) and in case $n=3$ ({\it{trifocal tensor}}) (\cite{BBBT1}). One of the key tool for these results is the reduction of the projection matrices into a particular form which we will call canonical form.

\

We start with the case of $n=3$ projections $\mathbb{P}^k \dashrightarrow \mathbb{P}^{h_j}, j=1,2,3$.
			
Recall that the entries $\tau_{j_1,j_2,j_3}$  of $\mathcal{T}$ are $(k+1) \times (k+1)-$ minors of $ \left[
\begin{array}{c}
	P_1  \\
	\vdots  \\
	P_n  \\
\end{array} \right]$ obtained by choosing $\alpha_j$ rows of $P_j$. Hence a projective transformation in in the scene ${\mathbb{P}}^k = {\mathbb{P}}(\mathcal{V})$ has the only effect of multipliyng all the entries of the tensor by a constant (the determinant of the associated matrix), while a projective trasformation on one view ${\mathbb{P}}^{h_j} = {\mathbb{P}}(\mathcal{W}^*_j)$ induces, through the Pl\"ucker embedding, a projective trasformation on $V_j = \bigwedge^{\alpha_j}{{\mathcal{W}}_j^\star}.$

Therefore, from Remark \ref{invariance}, one easily deduces the following

\begin{remark}
	\label{invar_grass}
	
\

	\begin{itemize}
		\item the right action of $GL(k+1)$ on $P$ doesn't change  $\mathcal{T}$ (up to constants) 
		
		\
		
		\item the left action of $GL(h_1+1) \times GL(h_2+1) \times GL(h_3+1)$ on $P$ doesn't change neither rk$(\mathcal{T})$ nor F-rk$({\mathcal{T}}).$
		
	\end{itemize}
	
\end{remark}

Now we make a {\it{general position assumption}} on the cameras, namely that the span of each pair of centers of projection does'n intersect the third center. 

Denote by $L_1$,
$L_2$ and $L_3$ the vector spaces of dimension $h_1+1$, $h_2+1$
and $h_3+1$ respectively, spanned by the columns of ${P_1}^T$,
${P_2}^T$ and ${P_3}^T$ and by $\Lambda_1=\mathbb{P}(L_1)$,
$\Lambda_2=\mathbb{P}(L_2)$ and $\Lambda_3=\mathbb{P}(L_3)$.

For each triplet of distinct integers $r,s,t=1,2,3,$ we can consider the following integers:

\begin{align}
	&i_{r,s}=h_r+h_s+1-k;\label{numrel1}\\
	&i=h_1+h_2+h_3+1-2k;\label{numrel2} \\
	&j_{r,s}=i_{r,s}-i.\label{numrel3}.
\end{align}

\noindent The general position assumption implies, in particular, that for any
choice of a pair $r,s$, the span of $L_r$ and $L_s$ is the whole
$\mathbb{C}^{k+1},$ (which means that the two centers $C_r$
and $C_s$ do not intersect), moreover, applying Grassmann formula one sees that the three
numbers above have the following geometric meaning: $i_{r,s}= dim (L_r \cap
L_s) \geq 0$, for any choice of $r,s$ , $i= dim (L_1 \cap L_2 \cap
L_3) \geq 0$ and $j_{r,s}$ is the affine dimension of the center
$C_t$ i.e. $k-{h_t}=j_{rs}$ for $r,s,t=1,2,3.$

\noindent Hence one can choose bases for $L_1, L_2$ and $L_3$ so that: 
$$L_1 = <v_1, \dots, v_i, w_{1}, \dots, w_{j_{1,2}}, u_{1}, \dots, u_{j_{1,3}}>,$$
$$L_2 = <v_1, \dots, v_i, w_{1}, \dots, w_{j_{1,2}}, t_{1}, \dots, t_{j_{2,3}}>,$$
$$L_3 = <v_1, \dots, v_i, u_{1}, \dots, u_{j_{1,3}}, t_{1}, \dots, t_{j_{2,3}}>,$$
\noindent and that
$\mathcal{B} = \{v_1, \dots, v_i, w_{1},
	\dots, w_{j_{1,2}}, u_{1}, \dots, u_{j_{1,3}},t_{1}, \dots,
	t_{j_{2,3}}\}$ is a basis of $\check{\mathbb{C}}^{k+1}.$

\noindent Making use of the action of $GL(h_i+1)$ on the views
$i=1,2,3,$ one can transform the columns of ${P_1}^T, {P_2}^T,$ and ${P_3}^T$ so that they become, respectively:
$$[v_1, \dots, v_i, w_{1}, \dots, w_{j_{1,2}}, u_{1}, \dots, u_{j_{1,3}}],$$
$$[v_1, \dots, v_i, w_{1}, \dots, w_{j_{1,2}}, t_{1}, \dots, t_{j_{2,3}}],$$
$$[v_1, \dots, v_i, u_{1}, \dots, u_{j_{1,3}}, t_{1}, \dots, t_{j_{2,3}}].$$

\noindent and finally, via the action of $GL(k+1)$ on the scene, we can reduce $\mathcal{B}$ to the canonical basis.
With these transformations the matrix $[P_1^T | P_2^T | P_3^T]$  takes the following form ({\it{canonical form}}):
\begin{equation}
	\label{canmat_trif}
	[{P_1}_C^T | {P_2}_C^T | {P_3}_C^T]:=\left[%
	\begin{array}{ccc|ccc|ccc}
		I_i & \mathbf{0} & \mathbf{0} & I_i & \mathbf{0} & \mathbf{0} & I_i & \mathbf{0} & \mathbf{0} \\
		\mathbf{0} & I_{j_{1,2}} & \mathbf{0} & \mathbf{0} & I_{j_{1,2}} & \mathbf{0} & \mathbf{0} & \mathbf{0} & \mathbf{0}\\
		\mathbf{0} & \mathbf{0} & I_{j_{1,3}} & \mathbf{0} & \mathbf{0} & \mathbf{0} & \mathbf{0} & I_{j_{1,3}} & \mathbf{0}\\
		\mathbf{0} & \mathbf{0} & \mathbf{0} & \mathbf{0} & \mathbf{0}& I_{j_{2,3}} & \mathbf{0} & \mathbf{0} & I_{j_{2,3}}\\
	\end{array}%
	\right],
\end{equation}
		
\noindent $I_a$ denoting the identity matrix of order $a$ and $\mathbf{0}$ denoting a suitable zero-matrix.

\

In the case of $n=2$ projections $\mathbb{P}^k \dashrightarrow \mathbb{P}^{h_j}, j=1,2$, we have only to assume that the two centers of projection do not meet each other, and, with similar, but easier, arguments we obtain the following {\it{canonical form}} for projection matrices

\begin{equation}
	\label{canmat_bifoc}
	[{P_1}_C^T | {P_2}_C^T]:=\left[%
\begin{array}{cc|cc}
	I_i & \mathbf{0} & I_i & \mathbf{0} \\
	\mathbf{0} & I_{h_1+1-i} & \mathbf{0} & \mathbf{0} \\
	\mathbf{0} & \mathbf{0} & \mathbf{0} & I_{h_2+1-i} \\
\end{array}%
\right].
\end{equation}

\subsection{Bifocal and trifocal tensors: ranks and core in the general case}
\label{rk-core}

Both in the case of the bifocal and in the case of the trifocal tensor we will denote by $\mathcal{T}_C$ the Grassmann tensor corresponding to projection matrices ${P_1}_C$ and ${P_2}_C$ (for $n=2$), ${P_1}_C, {P_2}_C$ and ${P_3}_C$ (for $n=3$, resp.). 

The tensor $\mathcal{T}_C$ has an extremely simple form: all the column vectors of its flattenings are either zero-vectors or, up to a sign, vectors of the canonical bases so that one can reason about $\mathcal{T}_C$ through purely combinatorial arguments.

On the other hand, due to Remark \ref{invar_grass} we have $\rk{\mathcal{T}} = \rk{\mathcal{T}_C},$ so that, by performing some combinatorial computations, one gets the following Theorems \ref{genfundmatteo}  and  \ref{rk_trifoc}.

\begin{theorem} (\cite{tubbAMPA})
	\label{genfundmatteo} The bifocal tensor
	$\mathcal{T}$ for two projections $\mathbb{P}^k \dashrightarrow \mathbb{P}^{h_j}, j=1,2$, whose centers do not intersect each other, with profile 
	$(\alpha_1,\alpha_2),$ has rank:
	$$\rk{\mathcal{T}} =\binom{(h_1-\alpha_1+1)+(h_2-\alpha_2+1)}{h_1-\alpha_1+1}.$$
	
\end{theorem}

\begin{theorem} (\cite{BBBT1})
	\label{rk_trifoc}
	The trifocal tensor $\mathcal{T}$ for three projections  $\mathbb{P}^k \dashrightarrow \mathbb{P}^{h_j}, j=1,2, 3$,  such that the span of any two centers do not intersect the third, with profile $(\alpha_1,\alpha_2,\alpha_3),$ has rank:

	$$ rk(\mathcal{T})=  \sum_{a_2=0}^{j_{12}}\sum_{a_3=0}^{j_{13}}\sum_{b_3=0}^{j_{23}}\binom{j_{12}}{a_2}\binom{j_{13}}{a_3}\binom{j_{23}}{b_3}\binom{i}{\alpha_1-a_2-a_3}\binom{i-\alpha_1+a_2+a_3}{\alpha_2-j_{12}+a_2-b_3},$$				
	
	\noindent where, for $\{r,s,t\} = \{1,2,3\}$ we have put $j_{rs} =k-h_t$.				
\end{theorem}

\
The proof of Theorem \ref{genfundmatteo} provided in \cite{tubbAMPA} is derived through a completely different approach, i.e., via a geometric interpretation of a map associated to the matrix $\mathcal{T}$. It turns out to be a rational map $\Phi: G(h_1-\alpha_1,h_1) \dashrightarrow G(k-\alpha_1, h_2)$ whose image is a suitable Schubert
variety $\Omega$, so that $\rk{\mathcal{T}} =dim(<\Omega>)+1,$ $<\Omega>$ denoting the projective space spanned by $\Omega.$

\

The canonical form (\ref{canmat_trif}) of the projection matrices plays also a key role for the computation of the multilinear rank of $\mathcal{T}.$  Indeed, for the tensor in canonical form, one can show that the multilinear rank  is $(r_1,r_2,r_3)$, where $r_j = n_j - \nu_j$, and $\nu_j =$ number of zero rows of $j-$ flattening of $\mathcal{T}_C$.

The result turns out to be the following
			
\begin{theorem} (\cite{BBBT2})
	\label{fl_rk_trifoc}
	The trifocal tensor $\mathcal{T}$ for three projections  $\mathbb{P}^k \dashrightarrow \mathbb{P}^{h_j}, j=1,2, 3$,  such that the span of any two centers do not intersect the third, with profile $(\alpha_1,\alpha_2,\alpha_3),$ has multilinear rank $(r_1,r_2,r_3)$ where
	
	$$r_1 = n_1 -  \sum_{(a_1,a_2,a_3) \in A}\binom{i}{a_1}\binom{j_{1,2}}{a_2}\binom{j_{1,3}}{a_3},$$
	
	\noindent with $A =\{ (a_1,a_2, a_3) \}$ such that $a_u$ are non negative integers satisfying
	$\,\sum_u a_u = \alpha_1, a_1 \le i, \max(0,\alpha_1 -i-j_{1,t})  \leq a_s\leq j_{1,s} -\alpha_s-1, a_t\le j_{1,t}. $
	\noindent and $r_2, r_3$ are obtained similarly to $r_1$ taking into account the cyclic nature of the indexes ${1,2,3}$. 
	
\end{theorem}

\

Once again, it is the canonical form that provides us with a method to find the core $\mathcal{C}$  of the trifocal tensor $\mathcal{T}$. Indeed, as mentioned before, the columns of the $j$-th flattening of ${\mathcal T}_C$ are zero or elements of the canonical basis $\{\mathbf{e}_1, \dots ,\mathbf{e}_{n_j} \}$ of $\mathbb{C}^{n_j}$. As a consequence, one can easily show that a core tensor ${\mathcal C}_C$ of ${\mathcal T}_C$ is obtained from ${\mathcal T}_C$ simply by deleting all null faces in each of the three directions, and that the matrices $U_j$ such that 
$$(U^*_1,U^*_2 ,U^*_3) \cdot \mathcal{T}_C = \mathcal{C}_C;$$
$$(U_1,U_2,U_3) \cdot {\mathcal C}_C = \mathcal{T}_C$$ 
are the $n_j \times r_j$ matrices whose columns are the vectors $\mathbf{e}_{k_1}, \dots, \mathbf{e}_{k_{r_j}}, j=1,2,3.$

Now, to obtain the core of ${\mathcal T}$, we use an approach which is essentially the same as HOSVD \cite{del-dem-van} and \cite{vann-vand-mee}, but that turns out to be easier due to the particular form of the involved matrices. First of all one considers the $r_j \times r_j$ invertible matrix $B_j$ defined by $B_j=E_jD_j^{-1}$ where $D_j$ is the diagonal matrix with the singular values of $V_j^{-1}U_j$ and $E_j$ is the matrix whose columns are the eigenvectors of $(V_j^{-1}U_j)^*(V_j^{-1}U_j)$.
Then one defines the core of the tensor ${\mathcal T}$ to be  ${\mathcal C} = (B_1^{-1}, B_2^{-1}, B_3^{-1}) \cdot {\mathcal C}_C$. Finally, one introduces matrices $S_j=V_j^{-1}U_jB_j$ for $j=1,2,3$, which are semi-orthogonal and verify that ${\mathcal C}$ is a core of ${\mathcal T}$, because  ${\mathcal C}=(S_1^*, S_2^*, S_3^*)\cdot {\mathcal T}$ and ${\mathcal T}=(S_1, S_2, S_3)\cdot {\mathcal C}$, since the following diagram commutes:
\begin{equation}
	\xymatrix{
		{\mathcal T}  \ar[rr]^{(V_1, V_2, V_3)}  & & \ar[dd]^{(U_1^T, U_2^T, U_3^T)} {\mathcal T}^c \\
		& \\
		{\mathcal C}  \ar[uu]^{(S_1, S_2, S_3)}  & & {\mathcal C}^c \ar[ll]^{(B^{-1}_1, B^{-1}_2, B^{-1}_3)}
	}
\end{equation}

\subsection{Non-general case}
\label{non-gen}
			
In subsection \ref{rk-core} we have determined the rank, the multilinear rank and the core of the trifocal tensor under a general position assumption on the centers of projection. If this condition is not satisfied, no general statement can be made, since the rank, the multilinear rank (and hence also the core) of the tensor strongly depends on the mutual position of the centers.
			
\
			
Just to give an idea for the rank, consider the case of $3$ projections $\mathbb{P}^{4} \dashrightarrow \mathbb{P}^2.$ When the lines which are centers of projection are in general position, we know ([\ref{rk_trifoc}]) that rk$(\mathcal{T}) = 4,$ but (see \cite{BBBT1}):

\begin{itemize}
\item rk$(\mathcal{T}) = 5,$ when the centers lie in the same hyperplane, or
\item rk$(\mathcal{T}) = 2,$ when the centers span $\mathbb{P}^{4}$, but two of them have nonempty intersection, 
\item rk$(\mathcal{T}) = 4,$ when the centers lie in the same hyperplane and two of them have nonempty intersection.
\end{itemize}

The examples seen above make it clear that the rank can both go up and down when the position of the centers is specialized, hence they provide evidence of the fact that the rank of tensors is not semicontinuous, also in the contest of multiview tensors.
We recall that a tensor $\mathcal{K}$ has {\it border rank} $r$ if it is a limit of tensors of rank $r$ but is not a limit of tensors of rank $s$ for any $s < r$ and that, of course $\underline{R}(\mathcal{K})\le R(\mathcal{K}),$ $\underline{R}(\mathcal{K})$ denoting the border rank of $\mathcal{K}$.  Examples above also show that 
the border rank of a multiview tensor can be strictly less than its rank.

\section{the variety of Grassmann tensors}
\label{var-grass}

In the previous sections, we analyzed the properties of a given Grassmann tensor. In this section, the entries of the cameras will be treated as parameters, and we will explore the algebraic varieties described by the corresponding Grassmann tensors. 

As before, we consider a Grassmann tensors associated to a set of $n$ projections $\mathbb{P}^k \dashrightarrow \mathbb{P}^{h_j}, j=1, \dots,n$, and we fix a profile $(\alpha_1\dots\alpha_n).$
We will denote by $\mathcal{X}^{{(k,h_1 \dots h_n)}}_{(\alpha_1\dots\alpha_n)}$ the variety described by these tensors. Before giving the formal definition of such variety, we recall the state of the art, in the classical case.

\

In case $n=2$, Grassmann tensors are fundamental matrices and a very classical result is the following

\begin{theorem}
\label{var-fund}
The variety of fundamental matrices is
$$\mathcal{X}^{(3,2,2)}_{(2,2)} = \{M\in Mat_{3,3} | det(M) = 0\}, $$
\noindent $Mat_{3,3}$ denoting the space of complex $3 \times 3$ matrices.

\end{theorem}

Recently, some authors have dealt with the study of varieties of Grassmann tensors in the context of $3$ or $4$ projections $\mathbb{P}^3 \dashrightarrow \mathbb{P}^2$. Their focus has primarily been on determining the ideals of these varieties.

In particular 

\begin{itemize}
	\item $n=3$, A. Alzati and A. Tortora \cite{al-to2} have determined the equations of the {\it{the trifocal variety}} $\mathcal{X}^{(3,2,2,2)}_{(1,1,2)}$, and  C. Aholt and L.Oeding \cite{oe1} have determined its ideal and computed its Hilbert polynomial,
	
	\item $n=4$, L.Oeding \cite{oe2} has computed the ideal of the {\it{the quadrifocal variety}} $\mathcal{X}^{(3,2,2,2,2)}_{(1,1,1,1)}$ up to degree $8$ and partially in degree $9$.
\end{itemize}  
			
\
	
The authors, in collaboration with some co-authors (\cite{tubbAMPA},\cite{bebitu},\cite{bebitu2}) have begun a research primarily focused on  determining the geometric and birational properties of the variety of Grassmann tensors $\mathcal{X}^{{(k,h_1 \dots h_n)}}_{(\alpha_1\dots\alpha_n)}$ for projections between spaces of any dimension.

As we have seen in subsection (\ref{Grass-tens}) for a Grassmann tensor $\mathcal{T}$ we know that ist entries are some of the Pl\"ucker coordinates of the linear space represented by the matrix $P^T = [P_1^T \dots  P_n^T]$ and that $\mathcal{T} \in V_1 \otimes  V_2 \otimes \dots \otimes V_n,$ where $V_j = \bigwedge^{\alpha_j}{{\mathcal{W}}_j^\star},$ 

Because of this, it is natural to consider the Grassmannian $G(k, {\mathbb{P}}({{\mathcal{W}}_1^\star} \oplus \dots \oplus {{\mathcal{W}}_n^\star}))$  and the projection 
$$\pi: {\mathbb{P}}(\bigwedge^{k+1}({{\mathcal{W}}_1^\star} \oplus \dots \oplus {{\mathcal{W}}_n^\star})) \dashrightarrow {\mathbb{P}}(\bigwedge^{\alpha_1}{{\mathcal{W}}_1^\star} \otimes \dots \otimes \bigwedge^{\alpha_2}{{\mathcal{W}}_2^\star}).$$
			
By definition, the variety $\mathcal{X}^{{(k,h_1 \dots h_n)}}_{(\alpha_1\dots\alpha_n)}$ of Grassmann tensors for $n$ projections $\mathbb{P}^k \dashrightarrow \mathbb{P}^{h_j}, j=1, \dots,n$ with profile $(\alpha_1\dots\alpha_n)$ is
$$ \mathcal{X}^{{(k,h_1 \dots h_n)}}_{(\alpha_1\dots\alpha_n)} = {\overline{\pi(G(k, {\mathbb{P}}({{\mathcal{W}}_1^\star} \oplus \dots \oplus {{\mathcal{W}}_n^\star})))}},$$
\noindent $\overline{(\ )}$ denoting the Zariski closure.
			
By Theorem (\ref{fund_result}), $\pi$ induces a birational map
			
$$\Pi: \frac{G(k, {\mathbb{P}}({{\mathcal{W}}_1^\star} \oplus \dots \oplus {{\mathcal{W}}_n^\star}))}{({\mathbb{C}}^*\oplus \dots \oplus{\mathbb{C}}^*)/{\mathbb{C}}^*} \dashrightarrow \mathcal{X}^{{(k,h_1 \dots h_n)}}_{(\alpha_1\dots\alpha_n)}.$$
			
In particular, one has
\begin{equation}
	\label{dim} dim(\mathcal{X}^{(k,h_1 \dots h_n)}_{(\alpha_1\dots\alpha_n)}) = (k+1)(h_1+h_2 + \dots + h_n +n-k-1)-n +1.
\end{equation}

Notice that the dimension of $\mathcal{X}^{(k,h_1 \dots h_n)}_{(\alpha_1\dots\alpha_n)}$ is independent of the profile $(\alpha_1, \dots,\alpha_n)$: different profile only correspond to different projections of the Grassmannian.
				
\

In \cite{tubbAMPA}, \cite{bebitu}, the authors focus on the case $n=2$ and, in order to give a geometric description of the variety $ {\mathcal{X}}_{\alpha_1,\alpha_2}$ of  {\it{bifocal Grassmann tensors}} they prove the following results

\begin{itemize}
    \item [a)] ${\mathcal{X}}^{(k,h_1, h_2)}_{(h_1,h_2)} =  \{M\in Mat_{h_1+1,h_2+1} \ | \ rk(M) = 2\}$ (see \cite{tubbAMPA}).
    \item [b)] $ {\mathcal{X}}^{(k,h_1, h_2)}_{(\alpha_1,\alpha_2)}$ is birational to a homogeneous space with respect to the action of $GL(h_1+1) \times GL(h_2+1)$ (see \cite{bebitu}).
    \item [c)] ${\mathcal{X}}^{(k,h_1, h_2)}_{(\alpha_1,\alpha_2)}$ is endowed with a dominant rational map  
$\Psi:  {\mathcal{X}}_{\alpha_1,\alpha_2} \dashrightarrow G(i-1,{\mathbb{P}})({\mathcal{V}^*}))$ with fibres isomorphic to $PGL(i)$ (see \cite{bebitu}).

\end{itemize}

Part a) provides a generalization of Theorem \ref{var-fund}. Its proof follows easily from formula (\ref{dim}).

The key tool for the proof of part b) is again the canonical form. Indeed in the space of $ (h_1+h_2+2)\times (k+1)$ matrices one can consider the open subset $Z$ of matrices $P =  \left[
\begin{array}{c}
P_1  \\
P_2  \\
\end{array} \right]$, with $P, P_1, P_2$ of maximal rank and where $P_j$ has size $(h_j+1) \times (k+1), j = 1,2$ and notice that, as we have seen in Section \ref{rank-core} every $P \in  Z$ can be put in the canonical form (\ref{canmat_bifoc}).

The proof of part c) is much more technical and for a comprehensive understanding, we refer to \cite{bebitu}.

\section {Grassmann Tensors and critical loci}
\label{GTCL}
		
		As discussed in Section \ref{Grass-tens}, given 
		$n$ projections of a scene in $\mathbb{P}^{k}$ and a suitable number of corresponding points in the images $\mathbb{P}^{h_i}$ of the scene, it becomes possible first to recover the entries of a suitable Grassmann tensor $\mathcal{T}$, then the projection matrices, and finally the coordinates of the scene points. 
  
        Hence, in principle, achieving a successful reconstruction is possible given a sufficient number of views and a a sufficient number of corresponding points.
        This holds true in general, but it is possible to find two sets of scene points (non projectively equivalent) and  cameras that produce the same images in the views. This scenario, which even occurs in the classical setup of two projections from $\Pin{3}$ to $\Pin{2}$, prevents any possible reconstruction. 

        Such configurations of points in $\Pin{k}$ are referred to as {\it critical} and the loci they describe as {\it critical loci}.

To formalize the notion of critical configuration and loci, let us
suppose to have $n$ projections of a static scene, consisting
of a set of $N >> 0$ points $\bxj$ in $\Pin{k}$. The projected points on the image spaces are denoted with $\mathbf{x}_{ij}=P_i(\nbxj).$ 
In all this section we will work under the:
\bigskip

{\em First generality assumption}: we assume that the intersection of the centers $ C_1 \cap \dots \cap C_n $ is empty.
\medskip

\begin{definition}
\label{critconfm} Given $ n $ projections $ Q_i: \mathbb{P}^k \dashrightarrow \mathbb{P}^{h_i}$,
a set of points $\{ \mathbf{X}_1, \dots, \mathbf{X}_N\}$
in $\Pin{k}$ is said to be a \textit{critical
configuration} for projective reconstruction for $ Q_1, \dots, Q_n$ if
there exists another set of $ n $ projections $P_i : \mathbb{P}^k \dashrightarrow \mathbb{P}^{h_i} $
and another set
$\{\mathbf{Y}_1, \dots, \mathbf{Y}_N \} \subset \Pin{k}$,
non-projectively equivalent to $ \{ \mathbf{X}_1, \dots,
\mathbf{X}_N \},$ such that, for all $i = 1, \dots, n$ and $j = 1,
\dots, N$, we have $P_i(\mathbf{Y}_j) = Q_i(\mathbf{X}_j)$, up to
homography in the targets. The two sets $\bxj$ and
$\{\mathbf{Y}_j\}$ are called {\it conjugate critical
configurations}, with {\it associated conjugate} projections (or {\it associated couple} of $n$ projections)
$\{Q_i\}$ and $\{P_i\}$ 
\end{definition}

According to notations in Section \ref{CV}, the $n$ projections are identified by $n$ matrices $P_i$, $i=1,..., n$, of dimension $(h_i+1)\times (k+1)$ and maximal rank. 

The points in critical configurations fill an algebraic variety, called {\em critical locus}  $\mathcal{X}$, whose ideal can be obtained by making use of the Grassmann tensor introduced in Section \ref{Grass-tens}:

\begin{proposition} \label{prop-2-1}
The critical locus $ \mathcal{X} $ is an algebraic variety, whose ideal $ I(\mathcal{X})$ is generated by the maximal minors of the matrix \begin{equation} \label{matrix-M}  M^{P_1, \dots,
P_n}_{Q_1, \dots, Q_n} = \left( \begin{array}{ccccccc} P_1 &
Q_1(\mathbf{X}) & 0 & 0 & \dots & 0 & 0 \\ P_2 & 0 &
Q_2(\mathbf{X}) & 0 & \dots & 0 & 0 \\ \vdots & & & & & & \vdots
\\ P_n & 0 & 0 & 0 & \dots & 0 & Q_n(\mathbf{X}) \end{array}
\right).\end{equation} Then $\mathcal{X}$ is a determinantal variety,
containing the centers of the projections $ Q_j$'s.
\end{proposition}
\begin{proof}

Let $ L_j \subseteq \mathbb{P}^{h_j} $ be corresponding linear subspaces, of codimensions $ \alpha_i $ for $ i=1, \dots, n$, and $(\alpha_1, \dots, \alpha_n)$ the corresponding profile, as defined in Section \ref{Grass-tens}. 
We have seen that the Grassmann tensor $\mathcal{T}^{P_1,\dots,P_n}(L_1,
\dots, L_n)$ encodes the algebraic relations between corresponding
subspaces in the different views of the projections
$P_1,\dots,P_n$. Hence by definition of critical set, if
$\{\mathbf{X}_j,\mathbf{Y}_j\}$ are conjugate critical
configurations, then, for each $j$, the projections
$Q_1(\mathbf{X}_j), \dots, Q_n(\mathbf{X}_j)$ are corresponding
points not only for the projections $Q_1, \dots,Q_n,$ but for the
projections $P_1, \dots,P_n$, too. 

We choose a point $ \mathbf{X} $ in the critical locus.
If $ Q_i(\mathbf{X}) \in L_i$, for every $ i=1,
\dots, n$, then $\mathcal{T}^{P_1,\dots,P_n}(L_1, \dots, L_n) =
0$. The previous condition is fulfilled if $ L_i $ is spanned by $
Q_i(\mathbf{X}) $ and any other $ h_i-\alpha_i $ independent
points in $ \mathbb{P}^{h_i}$. So, we can suppose $$ S_i = \left(
\begin{array}{cccc} Q_i(\mathbf{X}) & \mathbf{x}_{i1} & \dots &
\mathbf{x}_{i,h_i-\alpha_i} \end{array} \right) = \left(
\begin{array}{cc} Q_i(\mathbf{X}) & S'_i \end{array} \right) $$
of maximal rank $ h_i-\alpha_i+1$, that is to say, $ S'_i $
is a general $ (h_i+1) \times (h_i-\alpha_i) $ matrix of rank $
h_i-\alpha_i$. Due to this choice, the matrix $ T^{P_1, \dots,
P_n}_{S_1, \dots, S_n} $ becomes
$$ T^{P_1, \dots,
P_n}_{S_1, \dots, S_n} = \left( \begin{array}{cccccccccc} P_1 &
Q_1(\mathbf{X}) & S'_1 & 0 & 0 & 0 & \dots & 0 & 0 & 0 \\ P_2 & 0
& 0 & Q_2(\mathbf{X}) & S'_2 & 0 & \dots & 0 & 0 & 0 \\ \vdots & &
& & & & & & & \vdots \\ P_n & 0 & 0 & 0 & 0 & 0 & \dots & 0 &
Q_n(\mathbf{X}) & S'_n \end{array} \right).$$ The determinant $
\det(T^{P_1,\dots, P_n}_{S_1, \dots, S_n}) $ is a sum of products of
maximal minors of $ S'_1, \dots, S'_n$, and
maximal minors of the matrix \begin{equation} M^{P_1, \dots,
P_n}_{Q_1, \dots, Q_n} = \left( \begin{array}{ccccccc} P_1 &
Q_1(\mathbf{X}) & 0 & 0 & \dots & 0 & 0 \\ P_2 & 0 &
Q_2(\mathbf{X}) & 0 & \dots & 0 & 0 \\ \vdots & & & & & & \vdots
\\ P_n & 0 & 0 & 0 & \dots & 0 & Q_n(\mathbf{X}) \end{array}
\right).\end{equation}
Such a matrix is
a $ (n+ \displaystyle \sum_{i=1}^n h_i) \times (n+k+1) $ matrix,
the last $n$ columns of which are of linear forms, while the first
$ k+1 $ columns are of constants.

If we allow the profile to change, as $ \mathbf{X} $ is in
the critical locus independently from the profile, we get all the
possible maximal minors of $ M^{P_1, \dots, P_n}_{Q_1, \dots,
Q_n}$. 

To conclude the proof we only have to show that the center $ C'_j $ of $ Q_j $ is
contained in $ \mathcal{X}$ for every $ j=1, \dots, n$. $ C'_j $
is the zero locus of $ Q_j(\mathbf{X})$, and so $ M^{P_1, \dots,
P_n}_{Q_1, \dots, Q_n} $ drops rank at every point in $ C'_j$, for
each $ j$.
\end{proof}

\

We remark that, if $ \alpha_i \geq 1$, then some maximal minors of
$ M^{P_1, \dots, P_n}_{Q_1, \dots, Q_n} $ do not appear in $
\det(T^{P_1,\dots, P_n}_{S_1, \dots, S_n}) $ for whatever profile, and so
they should not be among the generators of $ I(\mathcal{X})$.  For this reason, we allow $ \alpha_j $ to be equal to $ 0 $ for some $ j$.
If this happens, the associated view does not impose any constrain
to the reconstruction problem, and so the effect of setting $
\alpha_j = 0 $ is to decrease the number of views.

\

In the following proposition we simplify the matrix used to generate the ideal $ I(\mathcal{X})$. In particular:
\begin{proposition}
\label{N-mat-cor}
$ I(\mathcal{X}) $ is generated by the maximal minors of a matrix
of linear forms and the expected dimension of $ \mathcal{X} $ is
\begin{equation} \label{exp-dim}
ed_{\mathcal{X}} = k - \left(1 + (n -k-1+ \sum_{i=1}^n h_i) -
n\right) = 2k - \sum_{i=1}^n h_i.
\end{equation}
Moreover, if $ \dim(\mathcal{X})
= ed_{\mathcal{X}}$,
\begin{equation} \label{exp-deg} \deg(\mathcal{X}) =
\binom{n-k-1+\sum_{i=1}^n h_i}{n-1}.
\end{equation} \bigskip
\end{proposition}
\begin{proof}	
From the first generality assumption, it follows
both that the first $ k+1 $ columns of $ M^{P_1, \dots, P_n}_{Q_1,
\dots, Q_n}$ are linearly independent, and that the linear forms
in the last $ n $ columns of the above matrix span a linear space
of dimension $ k+1 $ in $ R = {\mathbb  C}[x_0,\dots,
x_k]$, where $ \mathbb{P}^k = \mbox{Proj}(R)$. In fact, no point is common to either the
centers of the $ P_i$'s or of the $ Q_j$'s. \bigskip  

As in \cite{bnt1}, we write the matrix $ M^{P_1, \dots, P_n}_{Q_1, \dots, Q_n}$ as
the following block matrix $$ M^{P_1, \dots, P_n}_{Q_1, \dots,
Q_n} = \left( \begin{array}{cc} A & B \\ C & D \end{array} \right)
$$ where $ A $ is a $ (n -k-1+ \displaystyle \sum_{i=1}^n h_i) \times
(k+1)$ matrix, $ B $ is a $ (n -k-1+ \displaystyle \sum_{i=1}^n
h_i) \times n$ matrix, $ C $ is an order $ (k+1) $ square matrix,
and, finally, $ D $ is a $ (k+1) \times n $ matrix.

Due to the first generality assumption, $ \left( \begin{array}{c} A \\ C \end{array} \right) $ has rank $ k+1$. Hence, it contains an invertible square submatrix of order $ k+1$. Up to row exchanges, we can assume that the last $ k+1 $ rows give us the required invertible matrix, so we will assume $ C $ to be invertible.

By performing elementary operations on
columns and rows, we can reduce $ M^{P_1, \dots, P_n}_{Q_1, \dots,
Q_n} $ to the following easier form $$ \left(
\begin{array}{cc} 0 & N_{\mathcal{X}} \\ I_{k+1} & 0 \end{array} \right)
$$ where $ N_{\mathcal{X}} = B - AC^{-1}D $ is a $ (n -k-1+ \displaystyle \sum_{i=1}^n
h_i) \times n$ matrix of linear forms. Furthermore, the maximal
minors of $ M^{P_1, \dots, P_n}_{Q_1, \dots, Q_n}$ span the same
ideal as the maximal minors of $ N_{\mathcal{X}}$. Hence, we have
that $ I(\mathcal{X}) $ is generated by the maximal minors of $
N_{\mathcal{X}} = B - A C^{-1} D$. 
Since, as proven above, the critical locus $ \mathcal{X} $ is a determinantal
variety whose ideal is generated by the maximal minors of a matrix
of linear forms, the expected dimension of $ \mathcal{X} $ is
obtained from standard computations of the dimension of determinantal varieties. Moreover if $ \dim(\mathcal{X})
= ed_{\mathcal{X}}$,
from Porteous's formula
(\cite{acgh}, formula $4.2$), we get its degree.
\end{proof}
\bigskip

In the following sections we will give an overview on the state of art about the results on critical loci, according to the dimensions of the scene spaces and to the number of views. For these results we need to introduce the: 
\bigskip

{\em Second generality assumption:} we assume the projections 
$ \{P_i\}$ and $ \{Q_i\}$ to be general enough to guarantee that the critical locus $ \mathcal{X} $ has the expected dimension $ ed_{\mathcal{X}}=2k - \sum_{i=1}^n h_i$.
\bigskip 

Hence, from now on, every time we assert we are {\it in the general case}, we assume that both the generality assumptions hold. \medskip

\section{The case of one view from $\Pin{k}$ to $\Pin{2}$ }
\label{CALIB}

The calibration process of a single projection, whether in the classical case of a static scene in $\mathbb{P}^3$ or in its generalization to a dynamic scene in $\mathbb{P}^k$, equivalent to reconstructing the center of the projection, can be ambiguous. 
Hence we define the critical set for the calibration from $\mathbb{P}^k$ to $\mathbb{P}^h$ as
the set of points $\{\mathbf{X_i}\} \in \mathbb{P}^k$ having the same projection from two distinct centers. 

In this case, we we do not yet require the tensor machinery introduced in Section \ref{GTCL}. 

In the classical case of a single projection from $\mathbb{P}^3$ to $\mathbb{P}^2$, the critical locus has been proven to be a non-degenerate twisted rational cubic curve in $\mathbb{P}^3$ (see for example \cite{Hart-Zi2}, \cite{Hart-Ka}, \cite{Buch-92} and \cite{Buch}).

\

The generalization to the critical loci arising in the
reconstruction from a single view from $\Pin{k}, k \geq 4,$ to $\Pin{h}$ can be obtained as follows. Assume that $\{\mathbf{X_i}\}$ is a critical configuration for projective reconstruction from $1$-view from $\mathbb{P}^k$ to $\mathbb{P}^h$ with $h \geq 2$, as above, with associated projection matrices $P$ and $Q.$ As $P \cdot \mathbf{X_i} = \mu Q \cdot \mathbf{X_i},$ the associated critical locus is the determinantal variety defined by the $ 2 \times 2 $ minors of the $(h+1) \times 2$ - matrix $$ \left( \begin{array}{ccc} P(\mathbf{X}) & \vert & Q(\mathbf{X}) \end{array} \right).$$
Hence, if the critical locus is irreducible and reduced, it is a minimal degree variety in $\Pin{k}$ classified in \cite{Eisenbud-Harris}. 

In the case $h=2$, these critical loci are bundles $\mathbb{P}(E)$ over $\Pin{1},$ where $E =\oofp{1}{a_1} \oplus \oofp{1}{a_2}\oplus \dots \oplus \oofp{1}{a_{k-2}}$ is a rank $k-2$ vector bundle over $\Pin{1},$ $a_i \ge 0,$ $\sum_i a_i = 3,$ and $\mathbb{P}(E)$ is embedded by the tautological line bundle. In \cite{tubbCHAPTER}, the author presents a computer vision interpretation. If $k=4$ or $k=5$, the variety can either be smooth, if the null-spaces of $P$ and $Q$, that is, the centers of the two projections, do not intersect, or cones over a cubic surface with center $V$, if the two null spaces intersect at $V$. If $k \geq 6$ the centers of the two projections given by $P$ and $Q$ always intersect in a linear subspace $V$, hence all the null-spaces contain $V$. This implies that the variety is never smooth, but a cone with vertex $V$ over a suitable three dimensional variety.

\section{The case of $n \geq 2$ views from $\Pin{k}$ to $\Pin{h_i}$}
\label{MORE2}

An exhaustive analysis of critical loci in $\Pin{3}$
is presented in \cite{Hart-Ka}. The classical result (due to Kramer) regarding the criticality of a quadric surface in the case of $2$-views, is considered and generalized in the case of multiple views, to subvarieties obtained as suitable intersections of sets of critical quadrics. 

As the dimension $k$ of the ambient space increases, partial results, in the direction of determining the ideal of the critical variety for projections onto $\Pin{2}$ are presented in  \cite{tubbLAIA}, where a general framework for studying critical loci is proposed.

A complete classification of the {\it smooth} determinantal varieties which are critical loci for $n$ projections from $\Pin{k}$ to $\Pin{h_i}$ is contained in \cite{bnt2}. Moreover particular attention has been given to the case of three projections from $\Pin{4}$ to $\Pin{2},$ in which a Bordiga surface is obtained as essential component of the critical locus.

The results obtained in the case {\it smooth} and for the Bordiga surface are exposed in the following two sections.

\section {Smooth critical loci}
\label{sm}

For the results of this section we refer to \cite{bnt2} and we follow the notation of Section \ref{GTCL}.

Under the two generality assumptions recalled in the Section \ref{GTCL}, to classify the smooth critical loci in any dimension, i.e for $n$ projections from $\Pin{k}$ to $\Pin{h_i}$, we need the following preliminary facts. 

Let $ \mathcal{X} $ be the critical locus for a couple of $ n $ projections
$ P_1, \dots, P_n$ and $ Q_1, \dots, Q_n$ from $ \mathbb{P}^k$ to $ \mathbb{P}^{h_i}$, $ i=1, \dots, n$. Then we have: 
\begin{itemize}
    \item [1.] the following lower bound for $ \sum h_i$: 
    \begin{equation} k+1 \leq
\sum_{i=1}^n h_i. \end{equation}
\item [2.] if two centers of the projections $ Q_1, \dots, Q_n $ intersect, the critical locus is singular.
\item [3.] if $ n \geq 4 $ and $ \mathcal{X} $ has codimension $ c \geq 2 $, then, either $ \mathcal{X} $ is not irreducible, or is singular.
\item [4.] if $ n \geq 5 $ and $ \mathcal{X} $ has codimension $ c = 1$, then, either $ \mathcal{X} $ is not irreducible, or is singular.
\end{itemize}

Hence, smooth critical loci can appear only in the cases $ n=2 $ and $ n=3$, for every codimension $ c $ and $ n=4 $ for $ c=1$.
Now we analyse these remaining cases.

\subsection{The $ n = 2 $ views case}
\label{n2}

\begin{proposition}
\label{critical2minimal}
In the general case, the codimension $ c $ critical loci for two
views are minimal degree varieties.
\end{proposition}
When $ n=2$, we have $ h_1 + h_2 = k + c$ and $ k > h_2 \geq h_1 \geq c+1$. Moreover the ideal $ I(\mathcal{X}) $ is generated by the maximal minors of the matrix $ N_{\mathcal{X}}, $ [\ref{GTCL}], whose dimension is $(c+1) \times 2$ and hence is generated by quadrics. Moreover $deg(\mathcal{X})=c+1$ and hence $\mathcal{X}$ is a minimal degree variety in $\mathbb{P}^{k}$.

The generality assumptions implies that the minors of $ N_{\mathcal{X}} $ define a variety of the expected codimension $c$. From the classification of minimal degree varieties in \cite{Eisenbud-Harris}, we get that $
\mathcal{X} $ is singular as soon as $ k \geq 2(c+1)$. Hence, smooth irreducible varieties of minimal degree that can be critical loci are embedded in $ \mathbb{P}^k $ for $ c+2 \leq k \leq 2c+1$. For example, when the codimension is $1$, the critical locus for two projections from $ \mathbb{P}^3 $ to $ \mathbb{P}^2 $ is a quadric surface, and this is well--known; when $c=2$, the critical locus for two projections from $ \mathbb{P}^4 $ to $ \mathbb{P}^3 $ is a rational normal scroll; when $c=3$ the critical locus for two projections from $ \mathbb{P}^5 $ to $ \mathbb{P}^4 $ is either $ \mathbb{P}(\mathcal{O}_{\mathbb{P}^1}(2) \oplus \mathcal{O}_{\mathbb{P}^1}(2))$, or $ \mathbb{P}(\mathcal{O}_{\mathbb{P}^1}(1) \oplus \mathcal{O}_{\mathbb{P}^1}(3))$. \medskip

Conversely, starting from the matrix $ N_{\mathcal{V}}$ defining a generic minimal variety $\mathcal{V}$ of codimension $c$ in $ \mathbb{P}^k $, with elementary operations, it is possible to recover two pairs of projection matrices $P_1, P_2$ and $Q_1, Q_2$ for which $\mathcal{V}$ is a critical locus. Hence we have: 

\begin{proposition}
With the only exception of Veronese surfaces in $ \mathbb{P}^5$, every codimension $ c $ minimal degree variety embedded in $ \mathbb{P}^k $ with $ c+2 \leq k \leq 2c+1$, is the critical locus for a suitable pair of projections.
\end{proposition}

\subsection{The $ n = 3 $ views case}
\label{n3}

The classification result, in the case of three views, is contained in the following:

\begin{theorem} \label{n=3-critical-loci}
Let $ X \subseteq \mathbb{P}^k $ be a smooth codimension $ c $ variety. $ X $ is the critical locus for a suitable couple of three projections from $ \mathbb{P}^k $ if and only if either $ X \subseteq \mathbb{P}^2 $ is a cubic curve, or $ X \subseteq \mathbb{P}^3 $ is a cubic surface, or, finally, $ X \subseteq \mathbb{P}^4 $ is a Bordiga surface. In particular, $ c \leq 2$.
\end{theorem}
In the case under consideration, $ X $ is a codimension $ c $ smooth determinantal variety associated to a matrix $ N_{\mathcal{X}} $ of type $ (c+2) \times 3$. From the numerical bound proved before, we get that smooth determinantal varieties can be critical loci for two triples of projection only if embedded in a projective space $ \mathbb{P}^k $ with $ k = 2c $ or $ k = 2c+1$. The degree of such varieties is $\binom{c+2}2$, and consequentially,  thanks to classification results on smooth varieties with small invariants, we get the thesis.
For the converse problem, i.e. the problem of getting the above smooth varieties as critical loci for the reconstruction problem, we work quite similarly to the case of $2$ views: we consider a codimension $ c$, determinantal variety $ \mathcal{V} \subseteq \mathbb{P}^k$, with $ k = 2c$ or $ k = 2c+1$, whose defining ideal is generated by the $ 3 \times 3 $ minors of a $ (c+2) \times 3 $ matrix $ N_{\mathcal{V}}$ of linear forms. Performing elementary operations on $ N_{\mathcal{V}},$ it is possible to get again the two pairs of projection matrices for which the variety is critical.

\subsection{The $ n = 4 $ views case}
\label{n4}

In the case of $ 4 $ views, the codimension of the critical locus is one, otherwise the critical locus is either not irreducible or singular. In such a case, $ h_1 + h_2 + h_3 + h_4 = k+1$. On the other hand, a degree $ 4$, determinantal hypersurface is singular if embedded in $ \mathbb{P}^k $ with $ k \geq 4$. Hence, the only possible case is $ k=3$, and $ h_i = 1 $ for every $ i=1, \dots, 4$. 

\begin{proposition}
\label{7-1}
Let $ P_i, Q_i: \mathbb{P}^3 \to \mathbb{P}^1, i = 1, \dots, 4$, be two $4$--tuples of projections. Then, in the general case, the associated critical locus is a smooth quartic determinantal surface, containing four pairwise skew lines. 
\end{proposition}

We notice also that this quartic determinantal surface contains $4$ skew lines as centers of projections. It is known that the generic quartic surface in $ \mathbb{P}^3 $ is not determinantal, and that the locus of the determinantal ones is a divisor in $ \mathbb{P}^{34}$ (\cite{jessop-book}). Moreover, a generic quartic surface does not contain any line. 

Finally, we can consider a partial converse of the above Proposition:
 
\begin{proposition}
Let $ \ell_1, \dots, \ell_4 \subset \mathbb{P}^3 $ be $ 4 $ lines, pairwise skew. Then there exists a quartic determinantal surface $ S $ containing the $ 4 $ lines that is critical for two $4$--tuples of projections from $ \mathbb{P}^3 $ to $ \mathbb{P}^1$. The given lines are centers for the four projections $ Q_1, \dots, Q_4$.
\end{proposition}


\section{The Bordiga surface and reconstruction via lines}
\label{BORDIGA}

In this section we focus on the case of three projections  from $
\mathbb{P}^4 $ to $ \mathbb{P}^2$, where the
critical locus comes out to be a classical surface in $
\mathbb{P}^4$, specifically, the Bordiga surface, as hinted in the previous section, when the critical surface is smooth. 
The reason for dedicating a section to this specific case is the close connection between this critical locus and the reconstruction of a scene in $\mathbb{P}^3$ with projections via lines. The reconstruction using lines is particularly significant in the real word, since images of lines can be more easily and accurately detected and tracked than points. Furthermore the relationship between the critical locus for reconstruction via points (a Bordiga surface) and the one for reconstrucion via lines (a line congruence of bidegree $(3,6)$ and sectional genus $5$ in $ \mathbb{G}(1,3)$) is established through a birational map which is well known in the setting of classical algebraic geometry varieties. 

For the content of this section we refer to \cite{bnt1}.

\subsection{The Bordiga surface as critical locus in $\mathbb{P}^4$}

For convenience, we recall here the definition of the Bordiga surface as the embedding in $ \mathbb{P}^4 $ of the blow--up of $ \mathbb{P}^2 $ at $ 10 $ general points via the linear system of plane quartics through the points. In particular, let $Z=\{ p_1, \dots, p_{10}\} \subset \mathbb{P}^2
$ be the set of ten general points and let $ \widetilde{\mathbb{P}}^2
\stackrel{\pi}{\longrightarrow} \mathbb{P}^2 $ be the blow--up of
$\mathbb{P}^2$ in $ p_1, \dots, p_{10}$. The linear system $ \vert
4 \pi^\ast L - E_1 - \dots - E_{10} \vert $, where $ L $ is the
line divisor in $ \mathbb{P}^2$ and $ E_i $ is the exceptional
divisor associated to $ p_i$, $i=1,\dots, 10$, embeds $
\widetilde{\mathbb{P}}^2 $ in $ \mathbb{P}^4$, and becomes the
hyperplane divisor $ H $ of the image. If $ \phi :
\widetilde{\mathbb{P}}^2 \hookrightarrow \mathbb{P}^4 $ is the
embedding, $ B = \phi(\widetilde{\mathbb{P}}^2) $ is named
Bordiga surface. $ B $ is a smooth surface of degree $ 6 $ and
sectional genus $3$.

If $ Z \subseteq \mathbb{P}^2 $ is the set of $ 10 $ general points, and $ B \subseteq \mathbb{P}^4 $ the associated Bordiga surface, the ideal sheaves $ \mathcal{I}_Z $ and $ \mathcal{I}_B $ are defined from the following exact sequences: $$ 0 \to \mathcal{O}_{\mathbb{P}^2}^4(-5) \stackrel{N_Z}{\longrightarrow} \mathcal{O}_{\mathbb{P}^2}^5(-4) \to \mathcal{I}_Z \to 0 \quad \mbox{ and } \quad 0 \to \mathcal{O}_{\mathbb{P}^4}^3(-4) \stackrel{N_B}{\longrightarrow} \mathcal{O}_{\mathbb{P}^4}^4(-3) \to \mathcal{I}_B \to 0.$$ Moreover the matrices $ N_Z $ and $ N_B $ are not independent since it holds \begin{equation} \label{NZNB} (x_0, \dots, x_4) N_Z = (z_0, z_1, z_2) N_B^T \end{equation} where $ x_0, \dots, x_4 $ are coordinates in $ \mathbb{P}^4 $ and $ z_0, z_1, z_2 $ are coordinates in $ \mathbb{P}^2$.
Hence, more precisely, we can say that the Bordiga surface is the general element of the irreducible
component of the Hilbert scheme $
\mathcal{H}ilb_{p_B(t)}(\mathbb{P}^4) $ containing the codimension
$ 2$, ACM closed subschemes of $ \mathbb{P}^4 $ with Hilbert
polynomial $p_B(t) = (6t^2+2t+2)/2.$ \cite{ellingsrud}

According to the theoretical set up introduced in section \ref{GTCL} to determine the ideal $ I(\mathcal{X}) $ of a critical locus $\mathcal{X}$, in the case of of three projections  from $
\mathbb{P}^4 $ to $ \mathbb{P}^2$, we get that $ I(\mathcal{X}) $ is generated by the maximal minors of a $3 \times 4$ matrix $\mathcal{N_X}$, hence, under the usual generality assumption we have:
\begin{proposition} \label{ell-prop}
The general critical locus for projective reconstruction for three
views from $ \mathbb{P}^4 $ to $ \mathbb{P}^2 $ is in the
irreducible component of $ \mathcal{H}ilb_{p_B(t)}(\mathbb{P}^4) $
containing the Bordiga surfaces.
\end{proposition}

Indeed, we have a stronger result. Leveraging on an analysis of the matrix $\mathcal{N_X}$, on the minimal free resolutions of $ I(\mathcal{X}) $ and on the geometry of the Bordiga surface as blow-up of $\mathbb{P}^{2},$ we can prove that
\begin{proposition}
Let $ B $ be a Bordiga surface. Then $ B $ is a critical
locus for the projective reconstruction from three views from $
\mathbb{P}^4 $ to $ \mathbb{P}^2$, that is to say, there exist two couples of three projections
$ P_1,P_2,P_3$ and $ Q_1,Q_2,Q_3$ from $ \mathbb{P}^4$ to $ \mathbb{P}^2 $ such that the associated critical locus is $ B$.
\end{proposition}

\subsection{Projections of lines in $\mathbf{P}^3$ and critical loci}

In the previous sections, we addressed the reconstruction
problem for sets of points from their images. However, when dealing with projections of lines in $\mathbb{P}^3$ rather than points, it is necessary to consider a different set--up also for the reconstruction problem. (see for example \cite{Buch-92}, \cite{Buch}  and \cite{May-95}). In this context as well, there is a
natural notion of critical locus, consisting of lines in
$\Pin{3}.$
\begin{definition} \label{critconfm1} A set of lines
$\{\lambda_j\},$ $j=1,\dots,N,$ $N\gg 0,$ in $\Pin{3}$ is said to
be a \textit{critical configuration for projective reconstruction
of lines from three views} if there exist two collections of projections $\varphi_i$ and $\psi_i,$
$i=1,2,3,$ and a set of $N$ lines $\{\mu_j\}$ in $\Pin{3}$,
non-projectively equivalent to $ \{ \lambda_j \},$ such that, for
all $i$ and $j$, $\varphi_i(\lambda_j) = \psi_i(\mu_j)$, up to
homography in the (dual) image planes. The two sets
$\{\lambda_j\}$ and $\{\mu_j\}$ are called {\it conjugate critical
configurations}, with {\it associated conjugate} projections
$\{\varphi_i\}$ and $\{\psi_i\}$.
\end{definition}

In \cite{bnt1}, via an algebraic approach, we compute the
defining ideal of the critical locus for the reconstruction
via lines and we prove that:
\begin{theorem}
The critical locus for the reconstruction problem for a pair of
three projections $\varphi_i$, $\psi_j$, $i,j=1,2,3,$ of lines
from $ \mathbb{P}^3 $ to $ \mathbb{P}^2 $ is the union of a line
congruence of bi--degree $ (3,6) $ and sectional genus $ 5 $ and
the three $\alpha$--planes associated to the projection centers of
the $\varphi_i$'s. Moreover, each $ \alpha$--plane intersects the
line congruence along a degree $ 3 $ plane curve.
\end{theorem}

Once again for convenience, we recall the construction and some properties of the line congruence $ K $ of bi--degree $(3,6)$ and
sectional genus $5$ in the Grassmannian $ \mathbb{G}(1,3)$ of lines in $ \mathbb{P}^3$, following \cite{Verra}, \cite{arrondo}.

Let $ \widetilde{\mathbb{P}}^2 $
be the blow--up of $ \mathbb{P}^2 $ at $10$ general points $ p_1,
\dots, p_{10}$. The linear system $ \vert 7 \pi^\ast L - 2E_1 -
\dots - 2E_{10} \vert $ on $ \widetilde{\mathbb{P}}^2 $ embeds $
\widetilde{\mathbb{P}}^2 $ in
 $ \mathbb{G}(1,3) \subset \mathbb{P}^5$. For a suitable choice of the coordinates,
 $ \mathbb{G}(1,3) = V(x_0x_5 - x_1x_4 + x_2x_3)$. Let $ \psi: \widetilde{\mathbb{P}}^2 \to \mathbb{G}(1,3) $ be such an
 embedding, and let $ K = \psi(\widetilde{\mathbb{P}}^2)$ be its image. Then, $ K $ is a line congruence of bi--degree
 $(3,6)$ and sectional genus $5$. Furthermore, the converse is true; in other words, every such surface is the embedding of $ \widetilde{\mathbb{P}}^2 $ with
 $ \vert 7 \pi^\ast L - 2E_1 - \dots - 2E_{10} \vert $. In $ \mathbb{G}(1,3)$, the ideal sheaf
 $ \mathcal{I}_{K \vert \mathbb{G}(1,3)} $ is generated in degree $3$ and it holds
 \begin{equation}
 \label{arrondo-res} 0 \to \mathcal{O}^5_{\mathbb{G}(1,3)} \to \mathcal{E}_2(1)^3 \to \mathcal{I}_{K \vert \mathbb{G}(1,3)}(3) \to 0 \end{equation} where $ \mathcal{E}_2 $ is the rank $2$ vector bundle on $ \mathbb{G}(1,3) $ coming from the universal exact sequence $$ 0 \to \mathcal{E}_1 \to H^0(\mathcal{O}_{\mathbb{P}^3}(1)) \otimes \mathcal{O}_{\mathbb{G}(1,3)} \to \mathcal{E}_2(1) \to 0.$$
Moreover the minimal free resolution of $ I_K $ is:
\begin{equation} \label{resK} 0 \to T^5(-5) \to T^{12}(-4) \to \begin{array}{c} T(-2) \\ \oplus \\ T^7(-3) \end{array} \to I_K \to 0.\end{equation}
Hence, $ K $ is ACM with Hilbert polynomial $ p_K(t) =
(9t^2+t+2)/2$. So we get the following result, similar to Proposition \ref{ell-prop}. 
\begin{proposition}
    The line congruence of bi--degree $(3,6)$ and sectional genus $ 5 $ is the general element of the irreducible component of
the Hilbert scheme $ \mathcal{H}ilb_{p_K(t)}(\mathbb{P}^5) $ containing the codimension $3$,
ACM closed subschemes of $ \mathbb{P}^5 $ with Hilbert polynomial $ p_K(t) $ and resolution (\ref{resK}).
\end{proposition}

\subsection{Relationship between $B$ and $K$}

Given $Z=\{p_1, \dots, p_{10}\} \subset \mathbb{P}^2$, we have seen that the blow up $\widetilde{\mathbb{P}}^2 = Bl_Z(\mathbb{P}^2) $  is embedded both in $ \mathbb{P}^4
$ as a Bordiga surface $ B$, and in $ \mathbb{G}(1,3) $ as a
suitable line congruence $ K$.  In \cite{Verra}, it is proved that the map $ \theta=|\mathcal{O}_{\mathbb{P}^4}(2)-\Gamma|,$ where $\Gamma= \phi(l)$ is a rational normal quartic, makes the following diagram commutative and is biregular between $B$ and $K.$

	$$ \begin{tikzcd}[ampersand replacement=\&]  \& Bl_Z(\mathbb{P}^2)  \arrow[dl, "\phi_{|4H-Z|}"'] \arrow[dr,"\psi_{|7H-2Z|}"] \&  \\ \mathbb{P}^4 \supset B
		\arrow[rr, "\theta"]
		\& \&  K \subset \mathbb{P}^5 \end{tikzcd} $$

Recalling that $B \subset \mathbb{P}^4$ is critical locus for the reconstruction of points with three projections from $\mathbb{P}^4$ on $\mathbb{P}^2,$ and $K \subset G(1,3) \subset \mathbb{P}^5$ is the main component of the critical locus for the reconstruction of lines with three projections from $\mathbb{P}^3$ on $\mathbb{P}^2, $ we can restore this biregular correspondence in the setting of Computer Vision.
More precisely, the two reconstruction
problems considered in the previous sections are related each
other.  Given two triples of projections
from $ \mathbb{P}^4 $ to $ \mathbb{P}^2$, and the corresponding
critical locus $ \mathcal{X} \subset \mathbb{P}^4$, it is possible
to determine two triples of projections from $ \mathbb{P}^3 $ to $
\mathbb{P}^2 $ in such a way that the critical locus for the
reconstruction problem for lines is the union of three suitable $
\alpha$--planes and the image of $\mathcal{X}$ in $
\mathbb{G}(1,3) $ via the rational map $ \theta: \mathbb{P}^4 \to
\mathbb{G}(1,3)$ quoted in the diagram above. Furthermore, also the
converse holds. 

 \section{Dropping generality assumptions}
 \label{DEGENERATE}
 The results presented in the previous sections hold under the two generality assumptions introduced in Section \ref{GTCL}. When these generality assumptions are relaxed, some degenerate configurations of the center of projections and the corresponding degenerate critical loci naturally appear. 
 
 In the specific case of three projections  from $\mathbb{P}^4 $ to $ \mathbb{P}^2$, the minimal generators of the critical locus (Bordiga surface) are cubic polynomials obtained as maximal minors of a suitable matrix of dimension $4 \times 3$ of linear forms. The assumption of genericity is reflected in the fact that those maximal minors do not have any common factors.
Therefore, one is naturally led to consider $(n +1)\times n$ matrices of linear forms, whose minors have common factors. Matrices of type $(n +1)\times n$ that drop rank in codimension $2$ have been intensively studied within the framework of liaison theory, and in the context of commutative algebra, as a generalization of the Hilbert–Burch Theorem.

Conversely, matrices of type $(n +1) \times n$ of linear forms that drop rank in codimension $1$ do not seem to have been systematically studied. In \cite{bbnt} a classification of canonical forms of such matrices, over any field, for $n \leq 3$ is conducted and is utilized to study degenerate critical loci for suitable projections from $ \mathbb{P}^3$. In particular, the authors revisit the classical case of two projections from $\mathbb{P}^3 $ to $ \mathbb{P}^2$, refining results obtained originally by Hartley and Kahl in \cite{Hart-Ka} (see also \cite{BRAT} and \cite{BRAT3}).

 \section {Unified critical loci}
 \label{UNIFIED}
 
 For all the results of this sections refers to \cite{BNT-Fam1}.

 The definition of critical loci, as quoted and utilized in the preceding sections, focuses on the points ${\mathbf{X}}$ in $\mathbb{P}^k$ that cannot be reconstructed from their projections. Indeed, in the formulation of Definition \ref{critconfm}, another set of points ${\mathbf{Y}}$ appears, which is associated with the points ${\mathbf{X}}$ to determine their criticality. From a theoretical point of view, to highlight the symmetry between the two configurations of points ${\mathbf{X}}$ and ${\mathbf{Y}}$, in \cite{BNT-Fam1} the authors slightly modify Definition \ref{critconfm} as follows:
 
 \begin{definition} \label{critconfm2}
Given two sets of $ n $ projections $ Q_j, P_j: \mathbb{P}^k \dashrightarrow \mathbb{P}^{h_j}$,$j=1, \dots, n$,
two sets of points $\{ \mathbf{X}_1, \dots, \mathbf{X}_N\}$ and $\{ \mathbf{Y}_1, \dots, \mathbf{Y}_N\}$, $ N \gg 0$,
in $\mathbb{P}^{k}$, are said to be \textit{conjugate critical
configurations}, associated to the projections
$\{Q_1, \dots, Q_n\}$ and $\{P_1, \dots, P_n\}$ if, for all $i = 1, \dots, N$ and $j = 1,
\dots, n$, we have $Q_j(\mathbf{X}_i) = P_j(\mathbf{Y}_i)$.
\end{definition}

Notice that in Definition \ref{critconfm2}, we have symmetry between the projections $ Q_1, \dots, Q_n $ and $ P_1, \dots, P_n $ and from \ref{critconfm2}, it is evident that we can get a symmetric definition for the critical loci as well:

\begin{definition} \label{critical-loci2}
Given two sets of $ n $ projections $ Q_j, P_j : \mathbb{P}^k \dashrightarrow \mathbb{P}^{h_j}$, $j=1, \dots, n$, as above, the locus $ \mathcal{X} \subseteq \mathbb{P}^k $ ($ \mathcal{Y} \subseteq \mathbb{P}^k$, respectively) containing all possible critical configurations $ \{ \mathbf{X}_1, \dots, \mathbf{X}_N \} $ ($ \{ \mathbf{Y}_1, \dots, \mathbf{Y}_N \}$, respectively) is called \textit{critical locus} for the associated projections. Finally, we have the \textit{unified critical locus} $ \mathcal{U} \subseteq \mathbb{P}^k \times \mathbb{P}^k $ defined as the locus containing all the pairs $ (\mathbf{X}_i, \mathbf{Y}_i) $ of corresponding conjugate critical points.
\end{definition}

The motivation for introducing the unified critical locus $\mathcal{U}$ is to restore symmetry between the two sets of projections.  Such a symmetry is hidden when considering only the critical locus $ \mathcal{X}$ or $ \mathcal{Y}$. 

Generalising the construction of the ideal  $I(\mathcal{X}) $ (and $I(\mathcal{Y})$ by exchanging the projection matrices) as outlined in Section \ref{GTCL}, through the use of Grassmann tensors, we get the ideal of the unified critical locus $ \mathcal{U}$. Indeed, from its definition, it follows that $ \mathcal{U} $ is an algebraic variety in $ \mathbb{P}^k \times \mathbb{P}^k$.
Let us denote by $ M_{\mathcal{X}}$ and $ M_{\mathcal{Y}}$ the analogous of matrix (\ref{matrix-M}) for $\mathcal{X}$ and $\mathcal{Y}$ respectively.
\begin{proposition} \label{prop-id-unif-crit-locus}
The defining ideal $ I(\mathcal{U}) $ is bi--homogeneous generated by the maximal minors of $ M_{\mathcal{X}}$, $ M_{\mathcal{Y}} $ and by the $ 2 \times 2 $ minors of the matrices $$ \left( \begin{array}{ccc} P_j(\mathbf{Y}) & \vert & Q_j(\mathbf{X}) \end{array} \right), \quad j=1, \dots, n.$$
\end{proposition}

The unified critical locus $ \mathcal{U} $ is equipped with projections to the critical loci $ \mathcal{X}, \mathcal{Y}$, and so we have the following diagram
    \begin{equation}\label{main-cd} \begin{tikzcd}  & \mathcal{U} \ar[dl, "\pi_1"'] \ar[dr, "\pi_2"] & \\ \mathcal{X} & & \mathcal{Y} \end{tikzcd}. \end{equation}

If we consider the projections $ \pi_1 $ and $ \pi_2$, and the induced maps $ \pi_2 \circ \pi_1^{-1}: \pi_1(\mathcal{U}) \dashrightarrow \pi_2(\mathcal{U}) $ and $ \pi_1 \circ \pi_2^{-1}: \pi_2(\mathcal{U}) \dashrightarrow \pi_1(\mathcal{U})$, one might assume that $ \pi_1 $ and $ \pi_2 $ are dominant maps, and that $ \pi_2 \circ \pi_1^{-1}$, $ \pi_1 \circ \pi_2^{-1} $ are birational maps. On the contrary, this guess  does not always hold as shown in \cite{BNT-Fam1} with some examples. 
\begin{itemize}
\item[1.] The first example we produce concerns the classical setting of two pairs of projections from $ \mathbb{P}^3 $ to $ \mathbb{P}^2$. The twin critical loci $\mathcal{X}$ and $\mathcal{Y}$ are quadric surfaces. In this case, $ \pi_1 $ and $ \pi_2 $ are surjective, and $ \pi_2 \circ \pi_1^{-1} $ is a birational map. 
\item[2.] In a second example, we consider a couple of three projections from $ \mathbb{P}^3$, the first two to $ \mathbb{P}^2$  and the third to $ \mathbb{P}^1$. The twin critical loci $\mathcal{X}$ and $\mathcal{Y}$ are determinantal, with dimension 1 and degree $6$ in $\mathbb{P}^3.$ The unified critical locus $ \mathcal{U} \subseteq \mathbb{P}^3 \times \mathbb{P}^3 $ has dimension $ 1$ and degree $5$. In this case it possible to show that maps $ \pi_i $ are not dominant.
\item[3.] In a third example, concerning the construction of a singular Bordiga surface as degeneration of the critical loci for three projections from $\mathbb{P}^4$ to $\mathbb{P}^2$, we show that not every irreducible component of $ \mathcal{X} $ is necessarily birational to at least an irreducible component of $ \mathcal{Y}$.
\end{itemize}

Leveraging on these examples as well, we see that the description of the images  $\pi_1(\mathcal{U})$ and $\pi_2(\mathcal{U})$ as subvarieties of  $\mathcal{X}$ and $\mathcal{Y}$ can be quite intricate.

We can summarize the results on these maps with the help of the following diagram: 

$$ \begin{tikzcd}[ampersand replacement=\&]  \& \mathcal{U} \arrow[dl, "\pi_1"'] \arrow[dr,"\pi_2"] \&  \\ \mathcal{X} \supseteq \pi_1(\mathcal{U}) \arrow[rr, dashrightarrow, "\pi_2 \circ \pi_1^{-1}"] \& \& \pi_2(\mathcal{U}) \subseteq \mathcal{Y} \end{tikzcd}. $$
\bigskip	
	
\begin{itemize} 
	
	\item [-] in the general case, $\mathcal{X}-\{C_{Q_i}\} \subseteq \pi_1(U) \subseteq \mathcal{X}$, where $\{C_{Q_i}\}$ are the centers of the projection matrices $Q_i$.
		
	\item [-] if $ (\mathbf{X_0}, \mathbf{Y_0}) \in \mathcal{U}$ is a point such that $ \mathbf{X_0} $ is smooth on an irreducible component $ \mathcal{X}'$ of $ \mathcal{X}$, and $ \mathbf{Y_0} $ is smooth on an irreducible component $ \mathcal{Y}'$ of $ \mathcal{Y}$, then the map: $$ \pi_2 \circ \pi_1^{-1}: \mathcal{X}' \dashrightarrow \mathcal{Y}'$$ is a birational map.
		\end{itemize}
		
\section {Instability Phenomena}
\label{INSTABILITY}

From a realistic and practical point of view, it is unlikely that all the points of a scene or the centers of the cameras under consideration lie on a critical configuration. Therefore, one might draw the conclusion that critical sets have no effect in real life.

On the contrary, as highlighted in \cite{Hart-Ka} and \cite{tubbCHAPTER}, understanding the behavior of reconstruction for configurations of points near the critical locus is of practical relevance. In such circumstances, reconstruction solutions become extremely unstable, meaning that small perturbations in the data can lead to drastic changes in the output.
We conducted simulated experiments to illustrate this phenomenon, exploring various contexts and situations, in the case of one view, where we  conducted this experiment for the twisted cubic or the smooth, two-dimensional cubic scroll and in the case of multiple views, when the critical locus is a Bordiga surface (\cite{tubbCHAPTER}),  or an hypersurface (\cite{beMagri}) or a degenerate configurations (\cite{bbnt}).

A unified approach, in the case of $n \geq 2$ views, is outlined below from a theoretical point of view, was employed in these experiments:

\begin{itemize}
\item[1]{\it Generation of Critical Configurations} \\
Given two sets of projection matrices $\{P_i\}$ and $\{Q_i\}$, of the appropriate type, equations of the critical locus can be obtained (for example with the help of Maple). Equations are then solved to retrieve a suitable number of critical points $\{\mathbf{X_{j}}\}$ in $\Pin{k}.$ 

\

\item[2]{\it Perturbation of critical configurations}\\
Points $\{\mathbf{X_{j}}\}$ are then perturbed with a $k$-dimensional noise,
normally distributed, with zero mean, and with assigned standard
deviation $\sigma$, obtaining a new configuration
$\{\nbXi^{pert}\},$ which is close to being critical. This
configuration is projected via the original projection matrices $P_i$. 
The resulting images $\mathbf{x_{ij}} =
Pj\nbXi^{pert}$ are again perturbed with normally distributed
${h_i}$-dimensional noise with zero mean and fixed standard deviation to
obtain $\{\mathbf{x_{ij}}^{pert}\}.$

\

 \item[3]{\it Reconstruction} \\
The multifocal focal tensor corresponding to the true
 reconstruction, $T_P$ is computed from the projection matrices $P_i.$ using an algorithm implemented in Matlab.
 An estimated multifocal tensor $T$ is computed from
 $\{\mathbf{x_{ij}}^{pert}\}$ using a reconstruction algorithm implemented in Matlab as well.

\

\item[4]{\it Estimating instability}\\
As multifocal tensors are defined up to multiplication by a
non-zero constant, $T_P$ and $T$ are normalized. Using a suitable notion of distance, it is possible to estimate whether $T$ is close to $T_P,$ or
not. The above procedure is then
repeated many times for every fixed value of $\sigma.$

\

\item[5] {\it Results} \\
Results are then plotted in figures showing the frequency with which the reconstructed solution is close or far from the true solution $T_P,$ against the values of $\sigma$ utilized.
The results confirm that the larger the value of $\sigma,$ the stabler the solution gets, with standard deviation approaching zero.
\end{itemize}

\end{document}